\documentclass[a4paper,12pt]{amsart}

\usepackage[top=3cm, bottom=3cm, left=3.0cm, right=2.8cm]{geometry}

\usepackage{amsmath,amssymb,amsthm,graphicx,comment}
\usepackage[all]{xy}

\usepackage{lscape}
\usepackage{hyperref}
\usepackage{tikz}
\usetikzlibrary{arrows}
\usetikzlibrary{snakes}
\usetikzlibrary{shapes}
\usepackage{graphicx}
\usepackage{caption}
\usepackage{float}
\usepackage{subcaption}
\usetikzlibrary{calc,scopes,backgrounds,arrows}

\tikzset{%
	int/.style={fill=white,opacity=.85,pos=.5, inner sep=.1em,font=\scriptsize},%
	xvertex/.style={font=\scriptsize,minimum size=6pt, inner sep=1pt, outer sep=2pt},%
	pair/.style={coordinate,shape=circle,draw,scale=.3,fill=white, text=white},%
	solid/.style={coordinate,shape=circle,draw,scale=.3,fill=black},%
	sqr/.style={scale=.4,fill=white,shape=square,outer sep = 3pt},%
	vert/.style={scale=.4,fill=black,shape=circle,outer sep = 3pt},%
	wert/.style={scale=.4,draw,fill=white,shape=circle,outer sep = 3pt},%
	outline/.style={draw,line width=1.5mm,white},%
	sel/.style={draw,shape=circle,scale=2,color=red},%
	R/.style={inner sep=1,fill=gray!20},%
	bcut/.style={blue,dashed,thick},%
	rcut/.style={red,thick},%
}
\tikzset{
  mylabel/.style = {font=\footnotesize, midway, fill=white, anchor=center},
  mylabel2/.style = {font=\footnotesize, midway, fill=white, anchor=center,opacity=.8,pos=.5, inner sep=.1em}
}		
\pgfdeclarelayer{background}
\pgfsetlayers{background,main}

\author{Claire Amiot}
\author{Yvonne Grimeland}
\address{Institut Fourier, 100 rue des maths, 38402 Saint Martin d'H\`eres}
\email{claire.amiot@ujf-grenoble.fr}
\email{yvonne\_grimeland@hotmail.com}


\numberwithin{equation}{section}

\newtheorem{theorem}{Theorem}[section]
\newtheorem*{theorem*}{Theorem}
\newtheorem{lemma}[theorem]{Lemma}
\newtheorem{corollary}[theorem]{Corollary}
\newtheorem{proposition}[theorem]{Proposition}
\theoremstyle{remark}
\newtheorem{remark}[theorem]{Remark}
\newtheorem{example}[theorem]{Example}
\theoremstyle{definition}
\newtheorem{definition}[theorem]{Definition}

\newcommand{\jac}{{\operatorname{Jac}\nolimits}}
\newcommand{\End}{{\operatorname{End}\nolimits}}
\newcommand{\gr}{{\operatorname{gr}\nolimits}}

\newcommand{\add}{{\operatorname{add}\nolimits}}
\newcommand{\Hom}{{\operatorname{Hom}\nolimits}}

\newcommand{\smod}{\ensuremath{\underline{\mathrm{mod}}~\hat{\Lambda}}}

\newcommand{\new}[1]{{\color{blue}#1}}

\title{Derived invariants for Surface algebras}

\begin{document}
\begin{abstract}
In this paper we study the derived equivalences between \textit{surface algebras}, introduced by David-Roesler and Schiffler \cite{DS1}. Each surface algebra arises from a cut of an ideal triangulation of an unpunctured marked Riemann surface with boundary. A cut can be regarded as a grading on the Jacobian algebra of the quiver with potential $(Q,W)$ associated with the triangulation. 

Fixing a set $\epsilon$ of generators of the fundamental group of the surface $\pi_1(S)$, we associate to any cut $d$ a weight $w^\epsilon(d)\in\mathbb Z^{2g+b}$, where $g$ is the genus of $S$ and $b$ the number of boundary components. The main result of the paper asserts that the derived equivalence class of the surface algebra is determined by the corresponding weight $w^\epsilon(d)$ up to homeomorphism of the surface. Surface algebras are gentle and of global dimension $\leq 2$, and any surface algebras coming from the same surface $(S,M)$ are cluster equivalent, in the sense of \cite{AO-equ}. To prove that the weight is a derived invariant we strongly use results about cluster equivalent algebras from \cite{AO-equ}. 

Furthermore we also show that for surface algebras the invariant defined for gentle algebras by Avella-Alaminos and Geiss in \cite{AG}, is determined by the weight. 
\end{abstract}

\maketitle

%
%
%
\section{Introduction}
In this paper we study derived equivalence for the class of algebras called surface algebras. Surface algebras were introduced by David-Roesler and Schiffler in \cite{DS1}. 

Let $S$ be an oriented Riemann surface with non empty boundary, and $M$ be a set of marked points on the boundary of $S$. A quiver with potential $(Q^\Delta,W^\Delta)$ is associated to each ideal triangulation $\Delta$ of $(S,M)$ in \cite{ABCP}. The potential consists of the sum of the oriented three cycles in the quiver which correspond to internal triangles of the triangulation. The Jacobian algebra of $(Q^\Delta,W^\Delta)$, denoted $\jac(Q^\Delta,W^\Delta)$, is a finite dimensional gentle algebra \cite{ABCP}. 

A surface algebra is constructed from an ideal triangulation by introducing cuts in the internal triangles \cite{DS1}. This yields a grading on  $\jac(Q^\Delta,W^\Delta)$, and the surface algebra $\Lambda=(\Delta,d)$ associated with the grading is the subalgebra of degree zero of $\jac(Q^\Delta,W^\Delta)$. By \cite{DS1}, surface algebras are finite dimensional, gentle and have global dimension $\leq 2$.

A cluster category $\mathcal{C}_{(Q,W)}$ was associated to an arbitrary Jacobi-finite quiver with potential in \cite{clcat}. For a surface $(S,M)$ the category $\mathcal{C}_{(Q^\Delta,W^\Delta)}$ does not depend on the ideal triangulation $\Delta$ \cite{labardini, KellerYang}. The cluster category of a marked surface without punctures has been studied in \cite{BrZ}, where the Auslander-Reiten structure of $\mathcal{C}_{(S,M)}$ was described using a geometric model.

In \cite{clcat} a cluster category is also associated to any finite dimensional algebra of global dimension $2$ satisfying a certain homological property ($\tau_2$-finiteness). We show in section 3.2 that when $\Lambda=(S,M,\Delta, d)$ is a surface algebra then $\Lambda$ is $\tau_2$-finite and the cluster category of $\Lambda$ is equivalent to the cluster category of $(Q^\Delta,W^\Delta)$. Thus any two surface algebras coming from the same surface are cluster equivalent, in the sense of \cite{AO-equ}. 

Let $g$ be the genus of $S$, and $b$ the number of boundary components of $S$. We associate a weight $w^\epsilon(d)$ in $\mathbb Z^{2g+b}$  to each cut $d$ and set $\epsilon$ of generators of $\pi_1(S)$. Then using results in \cite{AO-equ} we are able to prove our main theorem;

\begin{theorem*}
Let $(S,M)$ be a surface which is not a disc and let $\Lambda$ and $\Lambda'$ be the surface algebras associated with $(\Delta,d)$ and $(\Delta',d')$, respectively. Then the following statements are equivalent;
\begin{itemize}
\item[(1)] $\mathcal{D}^b(\Lambda)\simeq\mathcal{D}^b(\Lambda')$.
\item[(2)] There exists an orientation preserving homeomorphism $ \Phi:S\rightarrow S$ with $w^\epsilon(\Delta,d)=w^\epsilon(\Phi^{-1}(\Delta'),d'\circ\Phi)$ (or equivalently  $w^\epsilon(\Delta,d)=w^{\Phi(\epsilon)}(\Delta',d')$).
\end{itemize}
\end{theorem*}

In the case where the surface is a disc with marked points on the boundary, it has already been shown in \cite{AO-tilde} that all surface algebras are derived equivalent.

Our result also generalizes a result in \cite{D2}, in which the author studies the case where $\Delta=\Delta'$ and the genus of $S$ is zero.

As mentioned above surface algebras are gentle, see \cite{DS1}. Avella-Alaminos and Geiss (abbreviated AG) introduced a derived invariant for gentle algebras in \cite{AG}. The invariant determines derived equivalence for gentle algebras with at most one cycle in the quiver. This invariant is calculated combinatorially using the quiver and relations of the algebra. In \cite{DS1}  the AG-invariant of any surface algebra is computed using the ideal triangulation of the surface and the cut.  Using this description, we show that the AG-invariant is determined by the weight of the corresponding cut. Our theorem then implies that in the genus zero case, the AG invariant determines the derived equivalence class of the surface algebra. In particular, it is easy to find derived equivalent surface algebras with more than one cycle in the quiver. Thus our result expands the class of algebras where the AG-invariant determines derived equivalence.

The paper is organized as follows; In section \ref{section1} we give the topological background and define the weight of a grading. In section \ref{sectionderivedinvariant} we give the relevant background on graded Jacobian algebras and generalized cluster categories before we state and prove our main theorem in subsection \ref{sectionmaintheorem}. In section \ref{genuszerosection} we examine the genus $0$ case and give examples. Finally, in section \ref{AGsection} we study the relationship between the weight and the AG-invariant.

%
%
%
%
\section{Surfaces without punctures and quivers}\label{section1}
In this section we first recall how to find a quiver from an ideal triangulation in section \ref{sectionidealtriangulation}. We give the definition of weight in section \ref{sectionquivers}, and in section \ref{secFlipMut} we look at how mutation affects the weight.
\subsection{Ideal triangulations}\label{sectionidealtriangulation}
We recall here some definitions of \cite{FST}.

Let $(S,M)$ be a connected oriented marked surface of genus $g$ with non-empty boundary. From now on and unless otherwise stated, we assume that $S$ is not a disc. 

The boundary $\partial S=B_1\cup B_2\cup \ldots B_b$ is the disjoint union of $b$ components $B_i$ each of which is homeomorphic to $\mathbb{S}^1$.  We assume that the set of marked points $M$ is a subset of $\partial S$ and that $\sharp (M\cap B_i)=p_i\geq 1$, and we write $p=\sum_i p_i=\sharp M$.

An \emph{arc} $\gamma$ in $(S,M)$ is a simple curve in $S$ (=without self intersection) with endpoints in $M$ which does not cut out a monogon or a digon. Such an arc is considered up to isotopy inside the class of such curves. If an arc $\gamma$ cuts out an $n$-gon with $n\geq 3$ we say that the arc is \emph{homotopic to the boundary}.

Two arcs are \emph{compatible} if there exist curves in their isotopy classes which do not intersect. An \emph{ideal triangulation} is a maximal collection of compatible arcs. The arcs of an ideal triangulation cut the surface $S$ into \emph{triangles}. Two (or three) of the vertices of a triangle may coincide. But since the marked points are on the boundary, the three sides are distinct from each other. Since the surface is oriented, each triangle, hence each angle of a triangle, inherits an orientation. If such an oriented angle $\alpha$ is an angle between two arcs, we denote by $s(\alpha)$ and $t(\alpha)$ the arcs adjacent to that angle so that $\alpha=(s(\alpha),t(\alpha))$. Note that since there is no punctures we cannot have angles $\alpha$ and $\alpha'$ with $s(\alpha)=t(\alpha')$ and $t(\alpha)=s(\alpha')$, but we may have $\alpha\neq \alpha'$ with $s(\alpha)=s(\alpha')$ and $t(\alpha)=t(\alpha')$.

All ideal triangulations of $(S,M)$ have the same number of arcs $n=6g-6+3b+p$, and the same number of triangles $4g-4+2b+p$. We call a triangle \emph{internal} if it has three sides which are arcs (and not boundary segments). We distinguish 3 classes of triangles:
\begin{itemize}
\item the class of triangles \emph{homotopic to the boundary}, that are triangles with the three sides homotopic to the boundary.

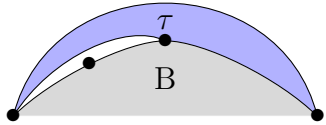
\[\scalebox{1}{
\begin{tikzpicture}[>=stealth,scale=1]

\draw[white,fill=gray!30] (0,0).. controls (0.5,0.5) and (1.5,1).. (2,1)..controls (2.5,1) and (3.5,0.5).. (4,0)--(0,0);

\draw (0,0).. controls (0.5,0.5) and (1.5,1).. (2,1);

\draw[fill=blue!30] (0,0).. controls (0.5,0.8) and (1.5, 1.25)..(2,1).. controls (2.5,1) and (3.5,0.5).. (4,0).. controls (3.5,2) and (0.5,2)..(0,0);

\node  at (0,0) {$\bullet$};
\node  at (2,1) {$\bullet$};
\node  at (4,0) {$\bullet$}; 
\node  at (1,0.7) {$\bullet$};

\node at (2,0.5) {B};

\node at (2,1.25) {$\tau$};

\end{tikzpicture}}\]
\item the triangles \emph{based on the boundary}, that are triangles with exactly one side homotopic to the boundary.

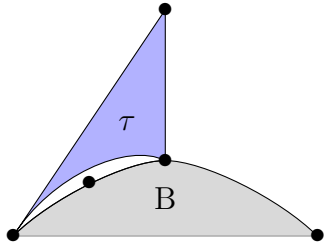
\[\scalebox{1}{
\begin{tikzpicture}[>=stealth,scale=1]

\draw[fill=gray!30] (0,0).. controls (0.5,0.5) and (1.5,1).. (2,1)..controls (2.5,1) and (3.5,0.5).. (4,0)--(0,0);

\draw[white] (0,0)--(4,0);

\draw (0,0).. controls (0.5,0.5) and (1.5,1).. (2,1);

\draw[fill=blue!30] (0,0).. controls (0.5,0.8) and (1.5, 1.25)..(2,1)--(2,3)--(0,0);

\node  at (0,0) {$\bullet$};
\node  at (2,1) {$\bullet$};
\node  at (4,0) {$\bullet$}; 
\node  at (1,0.7) {$\bullet$};
\node at (2,3) {$\bullet$};
\node at (2,0.5) {B};
\node at (1.5,1.5) {$\tau$};

\end{tikzpicture}}\]
\item the \emph{uncontractible} triangles which have three sides not homotopic to the boundary. (This third class is a subset of the internal triangles).

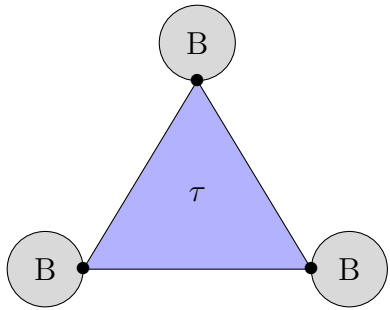
\[\scalebox{1}{
\begin{tikzpicture}[>=stealth,scale=1]

\draw[fill=gray!30] (0,0) circle (0.5);
\draw[fill=gray!30](4,0) circle (0.5);
\draw[fill=gray!30] (2,3) circle (0.5);

\draw[fill=blue!30] (0.5,0)--(3.5,0)--(2,2.5)--(0.5,0);

\node at (0.5,0) {$\bullet$};\node at (3.5,0) {$\bullet$};\node at (2,2.5) {$\bullet$};
\node at (0,0) {B}; \node at (4,0) {B}; \node at (2,3) {B}; 
\node at (2,1) {$\tau$};

\end{tikzpicture}}\]
\end{itemize}
Note that if a triangle has two sides homotopic to the boundary, then its third side also is.

\begin{proposition}\label{propuncontractible}
The number of uncontractible triangles is $4g-4+2b$. 
\end{proposition}

\begin{proof}
Let $\Delta$ be a triangulation of $(S,M)$. Contract each boundary component to a puncture. Then each triangle homotopic to the boundary is contracted into the point $B_i$. Each triangle based on the boundary is contracted into an arc. And each uncontractible triangle remains a triangle. This way we obtain an ideal triangulation of a surface of genus $g$ without boundary and with $b$ punctures.

\end{proof}

\begin{remark} Note that here we use the fact that $S$ is not a disc. In case of a disc, the number of uncontractible triangle is cleary $0\neq 4g-4+2=-2$.
\end{remark}

\subsection{Quivers and simplicial complexes}\label{sectionquivers}
For a reference on simplicial homology used in this subsection we refer the reader to \cite{topology}.

To an ideal triangulation $\Delta$ of $(S,M)$ we associate a quiver $Q=Q^\Delta$ as follows: 
\begin{itemize}
\item the set $Q_0$ of vertices is the set of arcs in $\Delta$.
\item the set of arrows $Q_1$ is  the set of angles: to an oriented angle $\alpha$ we associate the arrow $\alpha:s(\alpha)\to t(\alpha)$.
\item we also associate a set $Q_2$ of faces which is the set of $3$-cycles of $Q$ associated to each internal triangle $\tau$ of $\Delta$.
\end{itemize}

We define the following boundary complex:
\[ \xymatrix{C(\Delta)_\bullet :\ \mathbb{Z}Q_2\ar[r]^-{\partial_1} & \mathbb{Z}Q_1\ar[r]^-{\partial_0} & \mathbb{Z}Q_0}\]
where $\partial_1(abc)=a+b+c$ if $abc$ is an oriented $3$-cycle in $Q_2$, and $\partial_0(\alpha)=t(\alpha)-s(\alpha)$ if $\alpha:s(\alpha)\to t(\alpha)$ is an arrow in $Q_1$.

 \begin{lemma}
 $H_1(C(\Delta)_\bullet)\simeq \mathbb{Z}^{2g+b-1}.$
 \end{lemma}
 
 \begin{proof}
 Fix an arbitrary orientation of each arc of the triangulation $\Delta$. 
  We define sets $X_0$, $X_1$ and $X_2$ as follows. For each oriented arc $i$ of $\Delta$ we let $S_i$ and $E_i$  be two points on $i$ so that $i$ goes from $S_i$ to $E_i$. We define 
  $$X_0=\{ S_i,E_i\ i\in \Delta\}.$$ 
  For each angle $\alpha=(i,j)$ of $\Delta$ we associate a segment $\alpha'$ as follows:
  \[ \alpha'  = \begin{array}{ll} \ [E_i,S_j]   & \textrm{if the end point of }i\textrm{ is the starting point of }j \\ 
 \ [E_i,E_j]  & \textrm{if the end point of }i\textrm{ is the end point of }j  
  \\  \ [S_i,S_j]   & \textrm{if the starting point of }i\textrm{ is the starting point of }j 
  \\ \  [S_i,E_j]   & \textrm{if the starting point of }i\textrm{ is the end point of }j 
  \end{array}\]  
For each triangle with exactly one side being a boundary component, if $\alpha=(i,j)$ is the angle opposite to the boundary side, we define $\alpha''$ as the oriented segment linking the two points of $S_i,S_j,E_i,E_j$ which are not in $\alpha'$ and which has the same direction as $\alpha'$. 
We define \[X_1=\{[S_i,E_i],\ i\in \Delta\}\cup\{ \alpha', \alpha \textrm{ angle of }\Delta\}\cup \{\alpha'', \alpha \textrm{ opposite to a boundary side}\}.\]

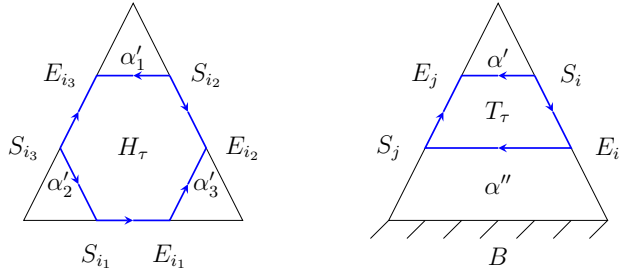
\[\scalebox{0.8}{
\begin{tikzpicture}[>=stealth,scale=1.2]
\draw (1,0)--(0,0)--(0.5,1);
\draw (1,2)--(1.5,3)--(2,2);
\draw (2.5,1)--(3,0)--(2,0);
\draw[->,blue, thick] (0.5,1)--(0.75,1.5);
\draw[blue, thick](0.75,1.5)--(1,2)--(1.5,2);
\draw[->,blue, thick](2,2)--(1.5,2);
 \draw[->,blue, thick](2,2)--(2.25,1.5);
 \draw[blue, thick](2.25, 1.5)--(2.5,1)--(2.25,0.5);
 \draw[->,blue, thick](1.5,0)--(2,0)--(2.25,0.5);
 \draw[->,blue, thick](0.75,0.5)--(1,0)--(1.5,0);
 \draw[->,blue, thick](0.5,1)--(0.75,0.5);
 \node at (1.5,1){$H_\tau$};
 \node at (0,1){$S_{i_3}$};\node at (0.5,2){$E_{i_3}$};\node at (2.5,2){$S_{i_2}$};\node at (3,1){$E_{i_2}$};\node at (1,-0.5){$S_{i_1}$};\node at (2,-0.5){$E_{i_1}$};
 \node at (1.5,2.25){$\alpha_1'$};\node at (2.5,0.5){$\alpha_3'$};\node at (0.5,0.5){$\alpha_2'$};

\draw (6,2)--(6.5,3)--(7,2);
\draw (7.5,1)--(8,0)--(5,0)--(5.5,1);
\draw[->,blue, thick] (6.5,1)--(5.5,1)--(5.75,1.5);\draw[blue, thick](5.75,1.5)--(6,2)--(6.5,2);\draw[->,blue, thick](7,2)--(6.5,2);
 \draw[->,blue, thick](7,2)--(7.25,1.5); \draw[->,blue, thick](7.25, 1.5)--(7.5,1)--(6.5,1);
 \node at (6.5,1.5){$T_\tau$};
\node at (5,1){$S_{j}$};\node at (5.5,2){$E_{j}$};\node at (7.5,2){$S_{i}$};\node at (8,1){$E_{i}$};
\node at (6.5,2.25){$\alpha'$};\node at (6.5,0.5){$\alpha''$};
\draw (5,0)--(4.75,-0.25);\draw (5.5,0)--(5.25,-0.25);\draw (6,0)--(5.75,-0.25);\draw (6.5,0)--(6.25,-0.25);\draw (7,0)--(6.75,-0.25);\draw (7.5,0)--(7.255,-0.25);\draw (8,0)--(7.75,-0.25);
\node at (6.5,-0.5){$B$};

\end{tikzpicture}}
\]

Then for each internal triangle $\tau=(\alpha_1,\alpha_2,\alpha_3)=(i_1,i_2,i_3)$ we get a hexagon $H_\tau=(\alpha_1', [S_{i_3},E_{i_3}] ,\alpha_2',[S_{i_1},E_{i_2}] ,\alpha_3',[S_{i_1},E_{i_1}] )$. And for each triangle with one side being a boundary component we get a trapezoid $T_\tau=(\alpha',[S_{i},E_{i}] ,\alpha'',[S_j,E_j])$ where $\alpha=(i,j)$ is the angle opposite to the boundary side. Note that since the arcs of $\Delta$ do not cross, then $H_\tau$ and $T_\tau$ are non degenerate hexagons and trapezoids respectively (their vertices are distinct from each other).
We define \[X_2=\{H_\tau, \tau \textrm{ internal triangle}\}\cup\{T_\tau, \tau \textrm{ triangle with one boundary side}\}.\]

Since each face in $X_2$ is naturally oriented, and since each segment of $X_1$ is oriented, we can define boundary maps $\partial_1:\mathbb{Z}X_2\to \mathbb{Z}X_1$ and $\partial_0:\mathbb{Z}X_1\to \mathbb{Z}X_0$ so that the following sequence is a complex of $\mathbb{Z}$-modules:
\[ \xymatrix{X_\bullet:\mathbb{Z}X_2\ar[r]^-{\partial_1} & \mathbb{Z}X_1\ar[r]^-{\partial _0} & \mathbb{Z}X_0}.\]
It is immediate that $(X_2,X_1,X_0)$ is a cell decomposition of a surface homeomorphic to the surface $S$. Therefore we get an isomorphism $H_1(X_\bullet)\simeq \mathbb{Z}^{2g+b-1}$.

We define a map $p:X_\bullet\to C_\bullet$ by $p(S_i)=p(E_i)=i$,  $p([E_i,S_i])=0$ for all $i$ in $\Delta$, $p(\alpha')=p(\alpha'')=\alpha$ for all angles $\alpha$ and   $p(H_\tau)=\tau\in Q_2$ for all internal triangles $\tau$  and $p(T_\tau)=0$ for all triangles $\tau$ with one boundary side. From the construction of $X_\bullet$, it is immediate that $p$ is a map of complexes which is surjective. Moreover its kernel is a complex $\xymatrix{Y_\bullet:\mathbb{Z}Y_2\ar[r]^{\partial_1} & \mathbb{Z}Y_1\ar[r]^{\partial _0} & \mathbb{Z}Y_0}$,  where  
$$Y_2=\{T_\tau, \tau \textrm{ triangle with one boundary side}\},\quad Y_0=\{E_i-S_i, i\in \Delta\}$$ $$\textrm{and }Y_1=\{[S_i,E_i],\ i\in \Delta\}\cup \{\alpha'-\alpha'', \alpha \textrm{ opposite to a boundary side}\}.$$  The exact sequence $\xymatrix{Y_\bullet \ar@{>->}[r] & X_\bullet\ar@{->>}[r] & C_\bullet}$ of complexes yields a long exact sequence in homology $$\xymatrix{\cdots\ar[r]& H_1(Y)\ar[r] &H_1(X)\ar[r] & H_1(C)\ar[r] & H_0(Y)\ar[r] & \cdots}.$$
The map $\partial_0:\mathbb Z Y_1\to \mathbb Z Y_0$ is clearly surjective, hence $H_0(Y)$ vanishes.
Moreover if $\alpha=(i,j)$ is an angle opposite to a boundary side, then we have $\alpha'-\alpha''=\partial_1(T_\tau)\pm[S_i,E_i] \pm[S_j,E_j]$, the sign depending on the orientation of $i$ and $j$. So for any $\gamma\in\mathbb{Z}Y_1$ there exists a $\Gamma\in \mathbb{Z}Y_2$ so that $\gamma-\partial_1(\Gamma)=\sum_i\lambda_i [S_i,E_i]$. Since the points $E_i$ and $S_i$ are isolated, then $\gamma\in{\rm Ker}\partial_0$ if and only if $\gamma\in{\rm Im}\partial_1$, that is  $H_1(Y_\bullet)$ vanishes, and $H_1(C_\bullet)=H_1(X_\bullet)\simeq \mathbb{Z}^{2g+b-1}$.
 \end{proof}
 
\begin{proposition}\label{equivstat1}
Let $\Delta$ and $\Delta'$ be ideal triangulations of $(S,M)$. Then the following are equivalent;
\begin{enumerate}
\item There is an isomorphism of $\mathbb{Z}$-complexes $C_\bullet(\Delta)\simeq C_\bullet(\Delta')$;
\item There exists an orientation preserving homeomorphism $\Phi:S\to S$ such that $\Phi(M)=M$ and $\Phi(\Delta)=\Delta'$. 
\end{enumerate}
 \end{proposition}
 
 \begin{proof}
 $(2)\Rightarrow(1)$ is immediate.
 
 $(1)\Rightarrow (2)$.  
 Let $\varphi:C(\Delta)\simeq C(\Delta')$ be an isomorphism.  It induces a bijection of quivers $Q^\Delta\simeq Q^{\Delta'}$. This can be viewed as a bijection $$\varphi:\{\textrm{arcs of }\Delta\}\to \{\textrm{arcs of }\Delta'\}$$ which respects angles, that is, $(i,j)$ is an oriented angle in $\Delta$ if and only if $(\varphi(i),\varphi(j))$ is an oriented angle in $\Delta'$.  Therefore, $\varphi$ induces a bijection between the set of points in $M$ incident with at least two arcs of $\Delta$ and the set of points in $M$ incident with at least two arcs of $\Delta'$. Moreover $(i,j,k)$ is an (oriented) internal triangle of $\Delta$ if and only if $(\varphi(i),\varphi(j),\varphi(k))$ is an (oriented) internal triangle in $\Delta'$. If $(i,j,s)$ is an oriented triangle of $\Delta$ with exactly one side $s$ being a boundary segment of $\Delta$, then $(\varphi(i),\varphi(j))$ is an oriented angle of $\Delta'$, and this angle does not belong to an internal triangle of $\Delta'$, so there exists a boundary segment $s'$ in $S$ such that $(\varphi(i),\varphi(j),s')$ is an oriented triangle in $\Delta'$. Let $(i,s,t)$ be an oriented triangle of $\Delta$ with $s$ and $t$ being boundary segments of $S$. Let $\tau$ be the other triangle of $\Delta$ containing $i$, then $\tau$ has at least two internal sides since $S$ is not a disc. Hence $\varphi(i)$ belongs to exactly one triangle of $\Delta'$ with at least one internal side. Thus $\varphi(i)$ belongs to one triangle of $\Delta'$ with two boundary sides. Therefore $\varphi$ induces a bijection between the points of $M$ and a bijection between the (oriented) boundary segments of $S$ which are compatible. Finally, for any internal triangle $\tau=(i,j,k)$ of $\Delta$,  there exists an orientation preserving homeomorphism $\Phi_\tau:\tau\simeq \varphi(\tau)$ which sends $i$ to $\varphi(i)$, $j$ to $\varphi(j)$ and $k$ to $\varphi(k)$. Since 
 the gluing data of $\Delta$ and $\Delta'$ are the same, there exists an orientation preserving homeomorphism $\Phi:S\to S$ such that $\Phi_{|_\tau}$ is homotopic to $\Phi_\tau$ for any triangle $\tau$ of $\Delta$. Therefore $\Phi$ is compatible with $\varphi$.
 \end{proof}


Let $\gamma$ be an oriented simple closed curve in $S$. Then in the isotopy class of $\gamma$ there exists a curve that intersects each arc of $\Delta$ transversally, and that does not intersect the same arc twice in succession. This implies that, locally on each triangle of $\Delta$, we can assume $\gamma$ to be as follows:
 
 \[\scalebox{1}
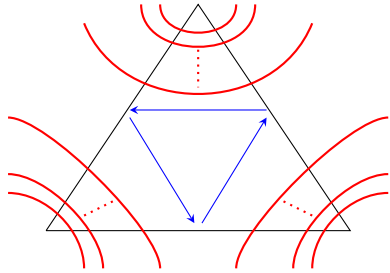
{
\begin{tikzpicture}[>=stealth,scale=1]
\draw (0,0)-- (2,3)--(4,0)--(0,0);
\draw[->,blue] (1.1, 1.5)--(1.95, 0.1);
\draw[->, blue] (2.05,0.1)--(2.9,1.5);
\draw[->,blue] (2.9,1.6)--(1.1, 1.6); 

\draw[red, thick] (-0.5,0.5)..controls(0,0.5) and (0.5,0)..(0.5,-0.5);
\draw[red, thick] (-0.5,1.5)..controls(0,1.5) and (1.5,0)..(1.5,-0.5);
\draw[red, thick] (-0.5,0.75)..controls(0,0.75) and (0.75,0)..(0.75,-0.5);
\draw[red,dotted, thick] (0.5,0.2)--(0.9,0.4);
\draw[red, thick] (2.5,-0.5)..controls(2.5,0) and (4,1.5)..(4.5,1.5);
\draw[red, thick] (3.5,-0.5)..controls(3.5,0) and (4,0.5)..(4.5,0.5);
\draw[red, thick] (3.25,-0.5)..controls(3.25,0) and (4,0.75)..(4.5,0.75);
\draw[red,dotted, thick] (3.5,0.2)--(3.1,0.4);

\draw[red, thick] (1.5,3)..controls (1.5,2.5) and (2.5,2.5)..(2.5,3);
\draw[red, thick] (1.25,3)..controls(1.25,2.25) and (2.75,2.25)..(2.75,3);
\draw[red, thick] (0.5,2.75)..controls(1,1.5) and (3,1.5)..(3.5,2.75);
\draw[red,dotted, thick] (2,2.4)--(2,1.9);

\end{tikzpicture}}\]
 
 Then we can associate to $\gamma$ an element $\bar{\gamma}$ in $\mathbb{Z}Q_1$ which is uniquely defined on the isotopy class of $\gamma$. More precisely, let $\tau$ be a triangle in $\Delta$ and $\alpha\in Q_1$ be in $\tau$, denote by $(\gamma\cap \alpha)^+$ the set of segments of $\gamma\cap \tau$ parallel to $\alpha$ and in the same direction, and by $(\gamma\cap\alpha)^-$ the set of segments $\gamma\cap \tau$ parallel to $\alpha$ and in the opposite direction. Then we have 
 \[\bar{\gamma}=\sum_{\alpha\in Q_1} (|(\gamma\cap\alpha)^+|-|(\gamma\cap\alpha)^-|)\alpha\]

 Let  $\epsilon=(a_1,b_1,\ldots, a_g,b_g,c_1,\ldots, c_{b})$ be a system of simple curves generating fundamental group $\pi_1(S)$ of $S$ such that $\pi_1(S)$ is the free group on $\epsilon$ modulo the relation \[\prod_{i=1}^g[a_i,b_i]=\prod_{i=1}^bc_i.\] 
 
 \begin{definition} 
 For $d\in \mathbb{Z}^{Q_1}=\Hom_\mathbb Z(\mathbb ZQ_1,\mathbb Z)$, we define the \emph{weight} of $d$ as 
 \[ w^\epsilon(d):=(d(\bar{a}_1),d(\bar{b}_1),\ldots, d(\bar{a}_g),d(\bar{b}_g),d(\bar{c}_1),\ldots, d(\bar{c}_{b}))\in\mathbb{Z}^{2g+b}.\] 
\end{definition}

\begin{lemma}\label{change of basis} Let $\Delta$ be an ideal triangulation of $(S,M)$, $\Phi:S\to S$ be an orientation preserving homeomorphism of $S$ with $\Phi(M)=M$ and $\epsilon=(a_1,\ldots, c_b)$ be a set of generators of $\pi_1(S)$. Then for any map $d\in \mathbb Z ^{Q_1}$ we have \[w^{\Phi^{-1}(\epsilon)}(\Delta,d)=w^{\epsilon}(\Phi(\Delta),d\circ\Phi^{-1}).\]
\end{lemma}
\begin{proof} Denote by $Q:=Q^\Delta$ the quiver corresponding to $\Delta$ and by $Q':=Q^{\Delta'}$ the quiver corresponding to $\Delta'=\Phi(\Delta)$.
For a simple closed curve $\gamma$ we denote by $\bar{\gamma}\in \mathbb Z Q_1$ its representative in $\mathbb Z Q_1$ and by $\bar{\gamma'}$ the corresponding representative in $\mathbb Z Q'_1$. The map $\Phi^{-1}:Q'_1\to Q_1$ extends by $\mathbb Z$-linearity to a map $\mathbb Z Q'_1\to \mathbb Z Q_1$ so we have   \[\Phi^{-1}(\bar{\gamma})=\sum_{\alpha'\in Q'_1} (|(\gamma\cap\alpha')^+|-|(\gamma\cap\alpha')^-|)\Phi^{-1}(\alpha')\in \mathbb Z Q_1.\] Now since $\Phi$ preserves the orientation of $S$, the curve $\gamma$ intersects the angle $\alpha'$  in the same direction (resp. in the opposite direction) if and only if the curve $\Phi^{-1}(\gamma)$ intersect the angle $\Phi^{-1}(\alpha')\in Q_1$ in the same direction (resp. in the opposite direction). So we get 
 \[\Phi^{-1}(\bar{\gamma})=\sum_{\alpha'\in Q'_1} (|(\Phi^{-1}(\gamma)\cap\Phi^{-1}(\alpha'))^+|-|(\Phi^{-1}(\gamma)\cap\Phi^{-1}(\alpha'))^-|)\Phi^{-1}(\alpha')=\overline{\Phi^{-1}(\gamma)}.\]
Hence we get the result. 
\end{proof}

\begin{remark}\label{isoremark}
Since $H_1(S,\mathbb Z)$ is the abelianization of $\pi_1(S)$, 
 then $(\bar{a}_1,\bar{b}_1,\ldots, \bar{c}_{b-1})\in (\mathbb{Z}Q_1)^{2g+b-1}$ is a basis of $H_1(C_\bullet)$. 
Denote by $$\xymatrix{C^\bullet:={\rm Hom}_\mathbb{Z}(C_\bullet,\mathbb{Z})=\mathbb{Z}^{Q_0}\ar[r]^-{\partial_0^*} & \mathbb{Z}^{Q_1}\ar[r]^{\partial_1^*} & \mathbb{Z}^{Q_2}}$$ the dual of $C_\bullet$. Then we also have ${\rm H}^1(C^{\bullet})\simeq \mathbb{Z}^{2g+b-1}$. The map which sends a degree map $d\in \mathbb Z ^{Q_1}$  to $(d(\bar{a}_1),d(\bar{b}_1),\ldots, d(\bar{a}_g),d(\bar{b}_g),d(\bar{c}_1),\ldots, d(\bar{c}_{b-1}))\in \mathbb Z ^{2g+b-1}$ induces such an isomorphism.
\end{remark}

\begin{definition}
A \emph{degree-$1$ map} on $Q$ is a map $d\in\mathbb{Z}^{Q_1}$ such that 
 $\partial_1^*(d)\in \mathbb Z ^{Q_2}$ is the constant function equal to $1$ on each $\tau\in Q_2$.

An \emph{admissible cut} on $Q$ is a degree-$1$ map such that ${\rm Im}(d)\subset\{0,1\}$ and for any $\alpha\in Q_1$, $d(\alpha)=1$ if and only if there exists $\tau\in Q_2$ with $\alpha\in \tau$.
\end{definition}

\begin{proposition}\label{propcb}
For any $i=1,\ldots, b$, denote by  $c_i$ a simple closed curve surrounding the boundary component $B_i$ and following the orientation of $S$.  Let $\Delta$ be an ideal triangulation of $(S,M)$ and $d$ a degree-$1$ map on $Q=Q^\Delta$, then we have $$\sum_{i=1}^bd(\bar{c}_i)=4g-4+2b.$$
\end{proposition}

\begin{proof}
Let $\alpha$ be in $Q_1$ and $\tau$ be the corresponding triangle of $\Delta$. Denote by $\gamma_1$ and $\gamma_2$ the arcs of $\Delta$ defining the angle $\alpha$, and by $\gamma_3$ the third side of $\tau$. For $j=1,2,3$  let $T_j\in B_{i_j}$ be the vertex of $\tau$ opposite to $\gamma_j$, and $B_{i_j}$ be the corresponding boundary component. Note that there exists $i$ such that $c_i\cap \alpha\neq \emptyset$ if and only if $\gamma_1$ and $\gamma_2$ are not homotopic to the boundary. Moreover $(c_i\cap \alpha)^+\neq\emptyset$ if and only if $c_i$ crosses $\gamma_1$ then $\gamma_2$ without crossing any arcs of $\Delta$ in between. Therefore $(c_i\cap \alpha)^+\neq \emptyset$ if and only if $i=i_3$ and in that case $|(c_{i_3}\cap \alpha)^+|=1$. On the other hand $(c_i\cap \alpha)^-\neq \emptyset$ if and only if $c_i$ crosses $\gamma_2$ and then $\gamma_1$ without crossing any arcs of $\Delta$ in between. This implies that $\gamma_3$ is homotopic to the boundary, and $\gamma_
 1$ and $\gamma_2$ are not. Therefore $(c_i\cap \alpha)^-\neq \emptyset$ if and only if $\tau$ is a triangle based on the boundary and $i=i_1=i_2$, and in that case $|(c_{i_1}\cap \alpha)^-|=1$.

Hence if $\tau$ is homotopic to the boundary we clearly have $\sum_{i}|(c_i\cap\alpha)^+|-|(c_i\cap\alpha)^-|=0$. If $\tau$ is based on the boundary we have $\sum_i |(c_i\cap\alpha)^+|-|(c_i\cap\alpha)^-|=|(c_{i_3}\cap\alpha)^+|-|(c_{i_1}\cap\alpha)^-|=0$. And if $\tau$ is an uncontractible triangle $\sum_i |(c_i\cap\alpha)^+|-|(c_i\cap\alpha)^-|=|(c_{i_3}\cap\alpha)^+|=1.$ 

Finally we get the following \begin{align*}\sum_id(\bar{c}_i) &= \sum_i\sum_{\alpha\in Q_1}(|(c_i\cap\alpha)^+|-|(c_i\cap\alpha)^-|))d(\alpha)\\ &=\sum_{\alpha\in Q_1}(\sum_i (|(c_i\cap\alpha)^+|-|(c_i\cap\alpha)^-|))d(\alpha) \\ & = \sum_{\alpha\in \tau \textrm{ uncontractible}} d(\alpha)= \sum_{\tau \textrm{ uncontractible}} d(\partial_1(\tau))\\ & = 4g-4+2b \textrm{ by Prop. \ref{propuncontractible}}.\end{align*}

\end{proof}

\begin{remark}
Note that in the case of a disc, $c_1$ is contractible,  and we still have $d(\bar{c}_1)=0=\sharp\{ \textrm{uncontractible triangles}\}$.  
\end{remark}

The next result immediately follows from Remark~\ref{isoremark} and Proposition~\ref{propcb}.
\begin{corollary}\label{corollary1}
Let $d$ and $d'$ be degree-$1$ map on $Q$, then the following are equivalent
\[w^\epsilon(d)=w^\epsilon(d')\Leftrightarrow d-d'\in {\rm Im}(\partial_0^*).\]
\end{corollary}

\subsection{Flips and mutations}\label{secFlipMut}

Let $\Delta$ be an ideal triangulation of $(S,M)$, and $i\in \Delta$ be an arc of $\Delta$. Then there exists a unique (up to homotopy) arc $\underline{i}$ of $(S,M)$ which is not homotopic to $i$ and such that $\mu_i(\Delta):=(\Delta-\{i\})\cup \{\underline{i}\}$ is an ideal triangulation of $(S,M)$. Such operation is called the \emph{flip} at $i$ of the triangulation $\Delta$. 

The corresponding operation on the quiver $Q_\Delta$ is called the \emph{mutation}. For $i\in Q_0$ the mutated quiver $\mu_i(Q)$ is constructed from $Q$ as follows:
\begin{enumerate}
\item Replace each arrow $a$ of $Q_1$ incident to $i$ by an arrow $a^*$ in the opposite direction in $\mu_i(Q)_1$;
\end{enumerate}
\noindent
For each $\xymatrix{j\ar[r]^a & i\ar[r]^b & k}$, 
\begin{enumerate}
\item[(2)] if there exists $c:k\to j$ such that $cba$ is in $Q_2$, then remove $c$ from $\mu_i(Q)_1$, and remove $cba$ from $\mu_i(Q)_2$;
\item[(3)] if there does not exist $c:k\to j$ such that $cba$ is in $Q_2$, then add $[ab]:j\to k$ in $\mu_i(Q)_1$ and add $a^*b^*[ab]$ to $\mu_i(Q)_2$. 
\end{enumerate}

\begin{proposition}\label{flipmut}
Let $\Delta$ be an ideal triangulation and $Q$ be the corresponding quiver. Let $i$ be in $\Delta$. Denote by $\Delta'=\mu_i(\Delta)$ and $Q'$ the corresponding quiver. Then there is an isomorphism of quiver $\mu_i(Q)\simeq Q'$ which is the identity on $Q_0$ and which induces an isomorphism $\mu_i(Q)_2\simeq Q_2'$.
\end{proposition}

\begin{proof}
The first assertion is shown in \cite{FST}. Let us show the second one. Let $a,b$ be arrows $\xymatrix{j\ar[r]^a & i\ar[r]^b & k}$, and denote by $\tau_a$ (resp. $\tau_b$) the triangle of $\Delta$ containing $a$ (resp. $b$). Then $a^*$ and $b^*$ are angles in the same triangle in $\mu_i(Q)_2$ if and only if there does not exist $c$ with $cba\in Q_2$, that is if and only if $\tau_a\neq \tau_b$. Since $\tau_a$ and $\tau_b$ are adjacent in $i$, there is an internal triangle in $\Delta'$ with sides $j$, $k$ and $i^*$ and containing the angles $a^*$ and $b^*$. 

\end{proof}

\begin{definition} \cite{AO-equ}
Let $\Delta$ be an ideal triangulation, and $d$ be a degree-$1$ map on $Q=Q^\Delta$. Then for any arc $i$ of $\Delta$ we define the \emph{left graded mutation} $\mu_i^L(\Delta,d)=(\Delta',d')$ of $(\Delta,d)$ at $i$ by $\Delta':=\mu_i(\Delta)$. The degree map $d'$ on $\mu_i(Q)$ is defined as follows:
\begin{enumerate}
\item for any $a:i\to j$ in $Q_1$, then $d'(a^*)=-d(a)$.
\item for any $b:j\to i$ in $Q_1$, then $d'(b^*)=1-d(b)$.
\item $d'([ab])=d(a)+d(b)$.
\end{enumerate} 
It is straightforward to check that $d'$ is again a degree-$1$ map on $\mu_i(Q)$.

One can also define the right graded mutation by $d'(a^*)=1-d(a)$ if $i$ is the source of $a$ and $d'(b^*)=-d(b)$ if $i$ is the target of $b$.
\end{definition}

\begin{lemma}\label{lemmamutation}
Let $\Delta$ be an ideal triangulation, and $d$ be a degree-$1$ map on $Q=Q_\Delta$. Then for any arc $i$ of $\Delta$ we have $w^\new{\epsilon}(\mu_i^L(\Delta,d))=w^\new{\epsilon}(\Delta,d)$.
\end{lemma}

\begin{proof}
Let $\gamma$ be any (oriented) simple closed curve on $S$. Denote by $\tau_1=(i,j_1,k_1)$ and $\tau_2=(i,j_2,k_2)$ the two triangles containing $i$ (Note that some of the arcs $j_1$, $j_2$, $k_1$, $k_2$ may coincide). If $\gamma$ does not intersect the rhombus $(j_1,k_1,j_2,k_2)$ then the assertion clearly holds. If $\gamma$ intersects the rhombus, without loss of generality we may assume that $\gamma$ first intersects $j_1$. Then locally there are only one of these three possibilities for $\bar{\gamma}$:

\[ c, \quad -b-f \textrm{ or}\quad -b+g \textrm{ with the following notations}\]

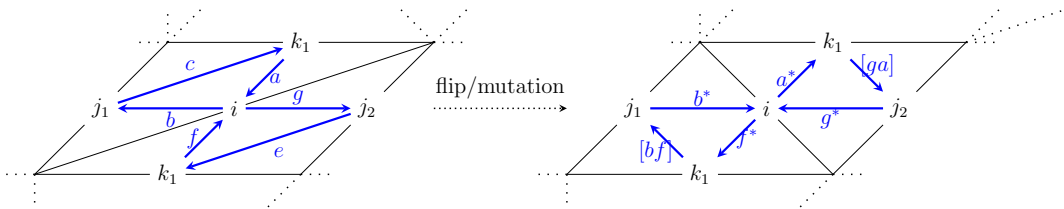
\[\scalebox{0.7}{
\begin{tikzpicture}[>=stealth,scale=2.5]
\node (A) at (0.5,0.5){$j_1$};
\node (B) at (2,1){$k_1$};
\node (C) at (2.5,0.5){$j_2$};
\node (D) at (1,0){$k_1$};
\node (I) at (1.5,0.5){$i$};
\draw  (0,0) -- (A)--(1,1);
\draw (1,1) -- (B)--(3,1);
\draw  (3,1) -- (C)--(2,0);
\draw (2,0) --(D)-- (0,0);
\draw (0,0)--(I)-- (3,1);
\draw [loosely dotted, thick] (0,0) -- (0,-0.25);
\draw [loosely dotted, thick] (0,0) --(-0.25,0);
\draw [loosely dotted, thick] (1,1) -- (1,1.25);
\draw [loosely dotted, thick] (1,1) --(0.75,1);
\draw [loosely dotted, thick] (1,1) -- (1.25,1.25);
\draw [loosely dotted, thick] (3,1) --(3.25,1);
\draw [loosely dotted, thick] (3,1) -- (3.25,1.25);
\draw [loosely dotted, thick] (3,1) --(2.75,1.25);
\draw [loosely dotted, thick] (2,0) -- (2.25,0);
\draw [loosely dotted, thick] (2,0) --(1.75,-0.25);

\draw[->,very thick, blue](I)--node[swap, yshift=-2mm]{$b$}(A);
\draw[->,very thick, blue](B)--node[swap, yshift=0mm, xshift=2mm]{$a$}(I);
\draw[->,very thick, blue](A)--node[swap, yshift=2mm, xshift=-2mm]{$c$}(B);
\draw[->,very thick, blue](I)--node[swap, yshift=2mm]{$g$}(C);
\draw[->,very thick, blue](C)--node[swap, yshift=-2mm, xshift=2mm]{$e$}(D);
\draw[->,very thick, blue](D)--node[swap, yshift=0mm, xshift=-2mm]{$f$}(I);

\draw[->, dotted, thick] (3,0.5)--node[swap, yshift=4mm]{flip/mutation}(4,0.5);

\node (A1) at (4.5,0.5){$j_1$};
\node (B1) at (6,1){$k_1$};
\node (C1) at (6.5,0.5){$j_2$};
\node (D1) at (5,0){$k_1$};
\node (I1) at (5.5,0.5){$i$};
\draw  (4,0) -- (A1)--(5,1);
\draw (5,1) -- (B1)--(7,1);
\draw  (7,1) -- (C1)--(6,0);
\draw (6,0) --(D1)-- (4,0);
\draw (5,1)--(I1)-- (6,0);
\draw [loosely dotted, thick] (4,0) -- (4,-0.25);
\draw [loosely dotted, thick] (4,0) --(3.75,0);
\draw [loosely dotted, thick] (5,1) -- (5,1.25);
\draw [loosely dotted, thick] (5,1) --(4.75,1);
\draw [loosely dotted, thick] (5,1) -- (5.25,1.25);
\draw [loosely dotted, thick] (7,1) --(7.25,1);
\draw [loosely dotted, thick] (7,1) -- (7.25,1.25);
\draw [loosely dotted, thick] (7,1) --(7.75,1.25);
\draw [loosely dotted, thick] (6,0) -- (6.25,0);
\draw [loosely dotted, thick] (6,0) --(5.75,-0.25);

\draw[<-,very thick, blue](I1)--node[swap, yshift=2mm]{$b^*$}(A1);
\draw[<-,very thick, blue](B1)--node[swap, yshift=0mm, xshift=-2mm]{$a^*$}(I1);
\draw[<-,very thick, blue](C1)--node[swap, yshift=2mm, xshift=2mm]{$[ga]$}(B1);
\draw[<-,very thick, blue](I1)--node[swap, yshift=-2mm]{$g^*$}(C1);
\draw[<-,very thick, blue](A1)--node[swap, yshift=-2mm, xshift=-2mm]{$[bf]$}(D1);
\draw[<-,very thick, blue](D1)--node[swap, yshift=0mm, xshift=2mm]{$f^*$}(I1);

\end{tikzpicture}}
\]
Denote by $d'$ the degree-$1$ map $\mu_i^L(d)$ on $\mu_i(Q)$.
Then one immediately checks:
\begin{align*} d'(b^*)+d'(a^*)=1-d(b)-d(a)=d(c) & \textrm{ since }d\textrm{ is a 1-degree map}\\
 -d'([bf])=-d(b)-d(f) & \\
d'(b^*)- d'(g^*)=-d(b)+d(g).
\end{align*}
Hence we get $d(\bar{\gamma})=\mu_i^L(d)(\bar{\gamma})$ and the result follows.
\end{proof}

%
%
%
%
\section{A Derived Invariant for Surface algebras} \label{sectionderivedinvariant}
In this section we prove our main theorem in section \ref{sectionmaintheorem}. In the next subsections we give the necessary background about graded Jacobian algebras in section \ref{sectiongradedjacobian}, generalized cluster categories in section~\ref{secgenCC} and results of \cite{AO-equ} about graded equivalent and cluster equivalent algebras in section~\ref{sectiongradedmut}. 

Note that by $k$ we denote in this section and throughout the rest of the paper an algebraically closed field.

\subsection{Graded Jacobian algebras}\label{sectiongradedjacobian} In this subsection we recall the definition of a graded Jacobian algebra, and we state and prove Proposition \ref{revisited2}, which plays an important role in the proof of the main Theorem \ref{maintheorem}.

Let $Q$ be a finite quiver. A potential on $Q$ is a linear combination of cycles in $Q$. 

A \emph{graded quiver with potential} (graded QP) $(Q,W,d)$ is the data of a finite quiver $Q$, of a potential $W$ on $Q$ and of a degree map $d:Q_1\to \mathbb Z$ such that $W$ is homogeneous of degree $1$.

For each arrow $a$ in $Q_1$ the \textit{cyclic derivative} $\partial_a$ with respect to $a$ is the unique linear map $\partial_a:kQ\to kQ$ which sends a path $\rho$ to the sum $\sum_{\rho=uav}vu$ taken over all decompositions of the path $\rho$. The associated \emph{Jacobian algebra} is  
$$\jac(Q,W):= k \hat{Q}/I(W), $$ 
where $k\hat{Q}$ is the completed path algebra, that is the completion of $kQ$ with respect to the ideal generated by the arrows, and $I(W):=\left\langle \partial_a W; a\in Q_1\right\rangle$ is the closure of the ideal generated by $\partial_a W$ for $a\in Q_1$, see \cite{DWZ}. Since $W$ is homogeneous of degree $1$, any relation in the jacobian ideal is homogeneous, so the degree map $d$ induces a grading on the algebra $\jac(Q,W)$. We denote by $\jac(Q,W,d)$ the corresponding $\mathbb Z$-graded algebra. 

Let $(S,M)$ be a connected oriented marked surface of genus $g$ with non-empty boundary, and let $Q=Q^{\Delta}$ be the quiver associated to the ideal triangulation $\Delta$ of $(S,M)$, cf. section 1.2. Then $W=\sum_{\tau\in Q_2}\tau$ is a potential on $Q$ and we associate to the triangulation the algebra $A(\Delta)=k\hat{Q}/I(W)=\jac(Q,W)$. If $d$ is a degree-$1$ map then $(Q^\Delta, W^\Delta, d)$ is a graded QP, so $A(\Delta,d)$ is naturally a $\mathbb Z$-graded algebra.

\begin{proposition}\cite{DS1}\cite{ABCP}\label{findim}
Let $(S,M)$ be a surface as in section~\ref{section1}, and $\Delta$ an ideal triangulation of $(S,M)$. Then the algebra $\jac(Q^\Delta,W^\Delta)$ is a finite dimensional gentle algebra. In particular the jacobian ideal $I(W^\Delta)$ is generated by paths of length $2$.
\end{proposition}

Moreover, from Lemma \ref{flipmut}, we obtain the following result.

\begin{proposition}\label{DWZmut} Let $(S,M)$ be a surface as in section~\ref{section1}, and $\Delta$ an ideal triangulation of $(S,M)$. For each arc $i$ of $\Delta$ there is a (graded) right equivalence in the sense of \cite{DWZ} (see \cite{AO-equ} for the graded version) $$\mu_i^L(Q^\Delta,W^\Delta,d)\sim (Q^{\mu_i(\Delta)}, W^{\mu_i(\Delta)}, \mu_i^L(d)),$$ where $\mu_i^L(Q^\Delta,W^\Delta,d)$ is the mutation of the graded QP defined in \cite{DWZ} (and in \cite{AO-equ} for the graded version). 
\end{proposition}

For a graded algebra $\Lambda:=\oplus_{p\in \mathbb Z}\Lambda_p$ we denote by $\gr \Lambda$ the category of finitely generated graded modules over $\Lambda$. If $M=\oplus_{p\in  \mathbb Z} M_p$ is a graded module, we denote by $M\left\langle q\right\rangle$ the graded module $\oplus_{p\in \mathbb Z}M_{p+q}$. (The degree $p$ part of $M\left\langle q\right\rangle$ is $M_{p+q}$.)

\begin{definition}
Let $A=kQ/I$  and $A'=kQ'/I$ be basic finite dimensional algebras. Let $d:Q_1\to \mathbb{Z}$  and $d':Q'_1\to \mathbb Z$ be two maps inducing gradings on $A$ and $A'$. Then the graded algebras $(A,d)$ and $(A',d')$ are said to be \emph{graded equivalent} if there exist $r_i\in \mathbb Z$ for each $i\in Q_0$ and an isomorphism of graded algebras
\[ A'  \underset{\mathbb Z}\simeq \oplus_{p\in \mathbb Z} \Hom_{{\rm gr} A} (\oplus_{i\in Q_0} P_i\left\langle r_i\right\rangle,\oplus_{i\in Q_0} P_i\left\langle r_i+p\right\rangle )\] where $P_i$ is the indecomposable projective $A$-module associated with $i\in Q_0$. 
Note that this isomorphism implies that $A$ and $A'$ are isomorphic algebras. When $A'=A$ and when the above isomorphism induces the identity morphism on $A$, we say that the gradings $d$ and $d'$ are equivalent via the identity.
\end{definition}
The next result is then an easy consequence of this definition and corollary~\ref{corollary1}.
\begin{lemma}\label{DRlemma} Let $(S,M)$ be a connected oriented marked surface, let $\Delta$ be an ideal triangulation of $(S,M)$, and let $d$ and $d'$ be admissible cuts on $\jac(Q^{\Delta}, W^{\Delta})$. Then $A=\jac(Q^{\Delta},W^{\Delta},d)$ and $A'=\jac(Q^{\Delta},W^{\Delta},d')$ are graded equivalent via the identity if and only if $w^\epsilon(d)=w^\epsilon(d')$. 
\end{lemma}

\begin{proof}
By corollary \ref{corollary1} we have $w^\epsilon(d)=w^\epsilon(d')$ if and only if there exists $r:Q_0\to \mathbb Z$ such that $d-d'=\partial^*_0(r)$. This means that for each arrow $\alpha:i\to j$ in $Q_1$, $d(\alpha)=d'(\alpha)+r(j)-r(i)$. Hence a path from $i$ to $j$ in $Q$ has $d'$-degree $p$ if and only if it has $d$-degree $p+r(j)-r(i)$. That is we have $\Hom_{\gr A'}(P_i,P_j\langle p\rangle)=\Hom_{\gr A}(P_i,P_j\langle p+r(j)-r(i)\rangle)$, which is the definition of $d$ and $d'$ being equivalent via the identity. 
\end{proof}

The next proposition is crucial in the proof of the main theorem, it relates graded equivalence between jacobian algebras and weight.
\begin{proposition}\label{revisited2}
Let $(S,M)$ be as in section \ref{section1} and $\epsilon$ be a set of generators of $\pi_1(S)$. Let $(\Delta, d)$ and $(\Delta',d')$ be ideal triangulations with admissible cuts on $Q:=Q^{\Delta}$ and $Q':=Q^{\Delta'}$, respectively. Then the following statements are equivalent:
\begin{itemize}
\item[(1)] There exists an orientation preserving homeomorphism  $\Phi: S \rightarrow S$ with $\Phi(\Delta)=\Delta'$ and $w^\epsilon(\Delta, d)=w^\epsilon(\Delta,d'\circ \Phi)$ (or equivalently $w^\epsilon(\Delta,d)=w^{\Phi(\epsilon)}(\Delta',d')$).
\item[(2)] There is a graded equivalence $\jac(Q,W,d)\sim \jac(Q',W',d')$.
\end{itemize}
\end{proposition}
\begin{proof}
First note that the equality $w^{\epsilon}(\Phi^{-1}(\Delta'),d'\circ\Phi)=w^{\Phi(\epsilon)}(\Delta',d')$ comes from Lemma \ref{change of basis}.

$(1)\Rightarrow (2)$
By Proposition \ref{equivstat1} the first part of statement $(1)$ is equivalent to the following; there is an isomorphism of $\mathbb Z$-complexes $\Phi: C_\bullet(\Delta)\rightarrow C_\bullet(\Delta')$. Assuming statement $(1)$, there is then a quiver isomorphism $\Phi:Q\rightarrow Q'$ and a bijection $\Phi: Q_2\rightarrow Q_2^{'}$, so we have 
$$ \Phi(W)=\Phi(\sum_{\tau\in Q_2}\tau)=\sum_{\tau\in Q_2}\Phi(\tau)=\sum_{\tau'\in Q_2^{'}}\tau'=W'.$$
Hence $\Phi$ induces an isomorphism of $\mathbb Z$-graded algebras
$$\jac(Q,W,d'\circ \Phi)\cong\jac(\Phi(Q), \Phi(W), (d'\circ \Phi)\circ \Phi^{-1})=\jac(Q', W', d').$$

Moreover by Lemma~\ref{DRlemma}, the equality $w^\epsilon(\Delta,d)=w^\epsilon(\Delta, d'\circ \Phi)$ means that the algebras $\jac(Q,W,d)$ and $\jac(Q,W,d'\circ\Phi)$ are graded equivalent via the identity. So we get $(2)$.

$(2)\Rightarrow (1)$
Now assume that there is a graded equivalence $$\jac(Q,W,d)\underset{\gr}{\sim} \jac(Q',W',d').$$ This graded equivalence means that there exists a map $r:Q_0\to \mathbb Z$ and an isomorphism of $\mathbb Z$-graded algebras 
$$\varphi:J=\jac(Q,W,d+\partial^*_0(r))\cong\jac(Q',W',d')=J'.$$ We denote by $\delta:=d+\partial^*_0(r)$ this new grading. 
The isomorphism $\varphi$ induces an isomorphism $\varphi:Q_0\to Q'_0$ between the primitive idempotents of the algebras. Let $i$ and $j$ be in $Q_0$, and let $p\in\mathbb Z$, then the number of arrows $\alpha:\varphi(i)\to \varphi(j)$ in $Q'$ with $d'(\alpha)=p$ is the same as the number of arrows $\beta:i\to j$ in $Q$ with $\delta(\beta)=p$. (Indeed we have $\dim {\rm Ext}^1_{\gr J}(S_j,S_i(p))=\dim {\rm Ext}^1_{\gr J'}(S_{\varphi(j)},S_{\varphi(i)}(p))$.) So there exists an isomorphism of quivers $\Phi_0:Q\to Q'$ compatible with $\varphi:Q_0\to Q'_0$ and such that $d'\circ\Phi_0=\delta$. 

We would like to construct an isomorphism of quiver $\Phi:Q\to Q'$ satisfying $\Phi(Q_2)=Q'_2$ and $d'\circ\Phi=\delta$. 
Let $n_0=\sharp\left\{\tau\in Q_2|\Phi_0(\tau)\in Q_2^{'}\right\}.$ 
If $n_0=\sharp Q_2$, then $\Phi_0$ induces a bijection $Q_2\rightarrow Q_2'$, so $\Phi=\Phi_0$ satisfies the conditions. 
If $n_0<|Q_2|$, then we construct a morphism of graded quivers $\Phi_1:(Q,\delta)\rightarrow (Q',d')$ with $n_1>n_0$, and conclude by induction. So let $\tau=\alpha\beta\gamma\in Q_2$ with $\Phi_0(\alpha\beta\gamma)\notin Q_2^{'}$. Now $\Phi_0(\alpha\beta\gamma)=\alpha'\beta'\gamma'$ is a $3$-cycle in $Q'$ since $\Phi_0$ is a quiver isomorphism. Since $\jac(Q',W')$ is a finite dimensional gentle algebra we must have that $\alpha'\beta'=0$ (or $\beta'\gamma'=0 $ or $\gamma'\alpha'=0 $). Then $\alpha'\beta'$ belongs to an element of $Q_2'$, so there exists $\gamma''\in Q_{1}^{'}$ such that $\alpha'\beta'\gamma''\in Q_2^{'}$ and with $\gamma'\neq \gamma''$ by assumption. The arrows $\gamma'$ and $\gamma''$ have the same source and target. So we define 
$$\Phi_1:Q \xrightarrow{\Phi_0} Q'\xrightarrow{\tau_{\gamma',\gamma''}}Q',$$
where $\tau_{\gamma',\gamma''}$ is the automorphism of $Q'$ exchanging $\gamma'$ and $\gamma''$. Since $d$ is a degree-$1$ map so is $\delta=d+\partial_0^*(r)$, hence we have \begin{align*}1 &=\delta(\alpha)+\delta(\beta)+\delta(\gamma)\\ &=d'\circ\Phi_0(\alpha)+d\circ\Phi_0(\beta)+d'\circ\Phi_0(\gamma)\\ & =d'(\alpha')+d'(\beta')+d'(\gamma').\end{align*}
Furthermore since $d'$ is a degree-$1$ map and since $\alpha'\beta'\gamma''$ is in $Q'_2$ we have $$1=d'(\alpha')+d'(\beta')+d'(\gamma'').$$ So $d'(\gamma')=d'(\gamma'')$, so $\Phi_1$ is compatible with the gradings, that is $d'\circ\Phi_1=\delta$.

We now proceed to show that $n_1 > n_0$. For any $3$-cycle $\tau=abc \in Q_2$ we consider the image $\Phi_0(abc)$ and whether it is in $Q_2^{'}$ or not. Since $\Phi_1= \tau_{\gamma',\gamma''}\circ \Phi_0$ it is easy to consider if also $\Phi_1(abc)$ is in $Q_2^{'}$. There are three cases to consider;
\begin{itemize}
\item[1)] If $\Phi_0(a),\Phi_0(b),\Phi_0(c)\neq \gamma',\gamma''$ then $\Phi_1 (abc)=\Phi_0(abc)$.
\item[2)] If $\Phi_0 (c)=\gamma'$ then $c=\gamma$ since $\Phi_0$ is an isomorphism of quivers. Furthermore, both $ab\gamma$ and $\alpha\beta\gamma$ are in $Q_2$, but since $\jac(Q^{\Delta},W^{\Delta})$ is gentle there can only be one $3$-cycle containing $\gamma$, so $a=\alpha$ and $b=\beta$. Thus $abc=\alpha\beta\gamma$ and $\Phi_0 (abc)\notin Q'_2$. However $\Phi_1(\alpha\beta\gamma)=\alpha'\beta'\gamma''\in Q_2^{'}$.
\item[3)] If $\Phi_0 (c)=\gamma''$ then $c\neq \gamma$ since $\gamma'\neq\gamma''$. Since any arrow can only be part of one $3$-cycle in $Q_2$ and $c\neq \gamma$ we have that $a\neq \alpha$ and $ b\neq \beta$.
For $\Phi_0 (c)=\gamma''$ we have that $\Phi_0 (abc)\in Q'_2$ if and only if $\Phi_0 (abc)=\alpha'\beta'\gamma''$. But we know from the above that $a\neq\alpha$ and $b\neq\beta$ so $\Phi_0 (abc)\neq \alpha'\beta'\gamma''$ and hence $\Phi_0 (abc)\notin Q'_2$ and furthermore, $\Phi_1(abc)\notin Q'_2$.
\end{itemize}
From this it is clear that $n_1\geq n_0+1$. Thus we conclude by induction that there exists an isomorphism $\Phi: Q\rightarrow Q'$ with $\Phi(Q_2)=Q'_2$ and such that $d'\circ \Phi=d''$. 

Therefore by Proposition~\ref{equivstat1} there exists an orientation preserving homeomorphism $\Phi:S\to S$ with $\Phi(\Delta)=\Delta'$. Moreover we have $$w(\Delta, d'\circ \Phi)=w(\Delta, \delta)=w(\Delta, d+\partial^*_0(r))=w(\Delta, d),$$ the last equality holding by corollary~\ref{corollary1}.

\end{proof}

\subsection{Generalized Cluster Categories}\label{secgenCC}
In this section we define surface algebras and prove that the cluster category of such an algebra is equivalent to the cluster category of the surface. We start by briefly recalling some results about cluster categories.

\subsubsection{The cluster category of a surface.}
The cluster category of a quiver with potential $\mathcal{C}_{(Q,W)}$ is defined in \cite{clcat}. Let $(S,M)$ be a surface with marked points on the boundary. Then for any triangulation $\Delta$ we can associate a cluster category $\mathcal{C}_{(Q^{\Delta},W^{\Delta})}$. This category is in fact independent of the triangulation by \cite{KellerYang,labardini}, and we hence denote it by $\mathcal{C}_{(S,M)}$. It is a $2$-Calabi-Yau triangulated category with cluster-tilting objects (see \cite{clcat} for more details and for definitions). Moreover it is shown in \cite{BrZ} that the set of isoclasses of cluster-tilting objects in $\mathcal{C}_{(S,M)}$ is in natural bijection with the set of ideal triangulations of $(S,M)$.

\subsubsection{The cluster category of an algebra of global dimension $2$.} Let $\Lambda$ be a finite dimensional algebra of global dimension $2$. We denote by $\mathbb{S}$ the Serre functor $-\overset{\textbf{L}}{\otimes}_\Lambda D \Lambda$ of the derived category $\mathcal{D}^b(\Lambda)$, and by $\mathbb{S}_2:=\mathbb{S}[-2]$. We say that $\Lambda$ is $\tau_2$-finite if the functor ${\rm H}^0(\mathbb{S}_2):\mod \Lambda \to \mathcal{D}^b(\Lambda)\to \mod \Lambda$ is nilpotent. For $\Lambda$ $\tau_2$-finite, a cluster category $\mathcal{C}_\Lambda$ is defined in \cite{clcat}. It is the triangulated hull (in the sense of \cite{kellerhull}) of the orbit category $\mathcal{D}^b(\Lambda)/\mathbb{S}_2$, so it comes naturally with a triangle functor:   
 $$ \pi: \mathcal{D}^b(\Lambda)\twoheadrightarrow \mathcal{D}^b(\Lambda)/\mathbb{S}_2\hookrightarrow \mathcal{C}_{\Lambda}.$$ Moreover the object $\pi(\Lambda)$ is a cluster-tilting object. 

Furthermore two $\tau_2$-finite algebras are called \emph{cluster-equivalent} algebras if there exists a triangle equivalence between their cluster categories.
 
\subsubsection{The cluster category of a surface algebra.}
We start by giving the definition of a surface algebra. 
\begin{definition} Let $(S,M)$ be a surface with marked points on the boundary, $\Delta$ an ideal triangulation and $d$ an admissible cut (cf section~\ref{section1}). The \emph{surface algebra} associated to $(\Delta, d)$ is the degree zero part of the graded jacobian algebra $\jac(Q^\Delta,W^\Delta,d)$ defined in the previous subsection.
\end{definition}

The next theorem gives the relation between the cluster category of a surface algebra and the cluster category of the surface. 
\begin{theorem}\label{equcat}
Let $\Lambda$ be the surface algebra associated to $(\Delta,d)$. Then $\Lambda$ is a finite dimensional algebra of global dimension $2$ and there is a triangle equivalence $\mathcal{C}_\Lambda \cong \mathcal{C}_{(S,M)}$ sending $\pi(\Lambda)$ to $\Delta$. 
\end{theorem}
 
\begin{proof}
From Proposition \ref{findim} and Theorem 3.5 of \cite{DS1} we have that $\Lambda$ is finite dimensional and has global dimension at most $2$. 
Furthermore, the quiver $Q^\Delta$ is obtained from the quiver of $\Lambda$ by adding an arrow $\gamma:i\to j$ for each zero relation $\alpha\beta:j\to\to j$ in $\Lambda$ and the potential $W^\Delta$ is the sum of  all $\alpha\beta\gamma$ where $\alpha\beta$ is a zero relation  in $\Lambda$ and where $\gamma$ is the arrow corresponding to $\alpha\beta$ (see \cite[Thm 3.5 (c)]{DS1}). Therefore the tensor algebra $T_\Lambda {\rm Ext}^2_\Lambda(D\Lambda,\Lambda)$ is isomorphic to $\jac(Q^\Delta,W^\Delta)$, so is finite dimensional. Hence $\Lambda$ is $\tau_2$-finite. Finally applying \cite[Theorem 6.11 a)]{keller} we obtain a triangle equivalence $\mathcal{C}_\Lambda\to \mathcal{C}_{(Q^\Delta,W^\Delta)}$ sending $\pi(\Lambda)$ to the canonical cluster-tilting object $T_{(Q^\Delta,W^\Delta)}$, that is, we obtain a triangle equivalence $\mathcal{C}_\Lambda\to \mathcal{C}_{(S,M)}$ sending $\pi(\Lambda)$ to $\Delta$. 
\end{proof}

\subsection{Graded mutation and derived equivalence}\label{sectiongradedmut}
In this subsection we recall some definitions and results from \cite{AO-equ} that are needed in the proof of the main theorem.

\begin{definition}\cite[Section 6.5]{AO-equ}
Let $\Lambda$ be a $\tau_2$-finite algebra of global dimension $\leq 2$ and let $T=T_1\oplus\cdots\oplus T_n$ be an object in $\mathcal{D}^b(\Lambda)$ such that $\pi(T)$ is a basic cluster-tilting object in $\mathcal{C}_{\Lambda}$. If $T_i$ is an indecomposable summand of $T\simeq T'\oplus T_i$, then define $T_i^L$ to be the cone in $\mathcal{D}^b(\Lambda)$ of the minimal left $\add\left\{\mathbb{S}_2^p T', p \in\mathbb{Z}\right\}$-approximation $u: T_i\rightarrow B$ of $T_i$. Denote by $\mu_i^L(T)$ the object $T_i^L\oplus T'$ and call it the \textit{left mutation} of $T$ at $T_i$.
\end{definition}
\noindent Note that $\pi(\mu_i^L T)=\mu_i(\pi T)$.

We have the following connection between left mutation in the derived category and the notion of left graded mutation defined in section \ref{secFlipMut}
\begin{theorem}\cite[Theorem 6.12]{AO-equ}\label{thm6.12}
Let $\Lambda=(S,M,\Delta,d)$ be a surface algebra and let $i_1,i_2,\ldots, i_l$ be a sequence of vertices of $Q^{\Delta}$ 
Denote by $T$ the object in $\mathcal{D}^b(\Lambda)$ defined by $T:=\mu_{i_l}^L\circ\cdots\circ\mu_{i_1}^L(\Lambda)$. Then there is an isomorphism of $\mathbb Z$-graded algebras
$$\oplus_{p\in\mathbb Z} \Hom_{\mathcal{D}^b(\Lambda)}(T,\mathbb{S}_2^{-p}T)\simeq \jac(\mu_{i_l}^L\circ\cdots\circ\mu_{i_1}^L(Q^{\Delta},W^{\Delta},d)) .$$
\end{theorem}

The next result shows that there is a close connection between graded equivalence and derived equivalence.
\begin{theorem}\cite[Cor. 6.14]{AO-tilde}\label{AO-derequ}
Let $\Lambda=(S,M,\Delta,d)$ and $\Lambda'=(S,M,\Delta',d')$ be surface algebras. If there exists a sequence $\sigma$ of mutations such that $\mu_{\sigma}^L(\Delta',d')=(\Delta, \delta)$ with $d$ and $\delta$ equivalent gradings then $\mathcal{D}^b\Lambda$ and $\mathcal{D}^b\Lambda'$ are derived equivalent.
\end{theorem}

Finally, the following result states an equivalent condition for derived equivalence of cluster-equivalent algebras.
\begin{theorem}~\cite[Cor. 5.4]{AO-tilde}\label{cor5.4}
Let $\Lambda$ and $\Lambda'$ be $\tau_2$-finite algebras of global dimension $2$ which are cluster equivalent. Denote by $\pi$ (resp. $\pi'$) the canonical functor $\mathcal{D}^b(\Lambda)\rightarrow \mathcal{C}_{\Lambda}$ (resp. $\mathcal{D}^b(\Lambda')\rightarrow \mathcal{C}_{\Lambda'}$ ).
Then the following two statements are equivalent:
\begin{itemize}
\item[(1)]$\mathcal{D}^b(\Lambda)$ and $\mathcal{D}^b(\Lambda')$ are derived equivalent.
\item[(2)] There exists an $\mathbb S_2$-equivalence between the categories $\pi'^{-1}(f\pi\Lambda)\subset\mathcal{D}^b(\Lambda')$ and $\pi^{-1}(\pi\Lambda)\subset\mathcal{D}^b(\Lambda)$ for some triangle equivalence $f:\mathcal{C}_{\Lambda}\rightarrow\mathcal{C}_{\Lambda'}$.
$$
\begin{tikzpicture}[bij/.style={above,sloped,inner sep=0.3pt}]
\node (1) at (0,0) {$\mathcal{C}_{\Lambda} $};
\node (2) at (2,0) {$\mathcal{C}_{\Lambda'}$};
\node (3) at (0,-1.5) {$ \mathcal{D}^b(\Lambda)$};
\node (4) at (2,-1.5) {$ \mathcal{D}^b(\Lambda')$};
\draw [->] (1) edge node[bij] {$\sim$} node[below]{$\scriptstyle f$} (2);
\draw [->] (3)--(1);
\draw [->] (4)--(2);
\end{tikzpicture}
$$

\end{itemize}

\end{theorem}
\subsection{Proof of the main theorem}\label{sectionmaintheorem}
In this subsection we state and prove our main result. 
\begin{theorem}\label{maintheorem} Let $(S,M)$ be a surface as in section 2 and let $\Lambda$ and $\Lambda'$ be the surface algebras associated with $(\Delta,d)$ and $(\Delta',d')$, respectively. Then the following statements are equivalent;
\begin{itemize}
\item[(1)] $\mathcal{D}^b(\Lambda)\simeq\mathcal{D}^b(\Lambda')$.
\item[(2)] There exists an orientation preserving homeomorphism $ \Phi:S\rightarrow S$ with $w^\epsilon(\Delta,d)=w^\epsilon(\Phi^{-1}(\Delta'),d'\circ\Phi)$ (or equivalently  $w^\epsilon(\Delta,d)=w^{\Phi(\epsilon)}(\Delta',d')$).
\end{itemize}
\end{theorem}
\begin{proof}
$(1)\Rightarrow (2)$ Assume that $\mathcal{D}^b(\Lambda)\simeq \mathcal{D}^b(\Lambda')$. This equivalence induces a triangle equivalence $\mathcal{C}_\Lambda\simeq\mathcal{C}_{\Lambda'}$.  Furthermore from Theorem \ref{equcat} we have that $\mathcal{C}_{\Lambda}\simeq\mathcal{C}\simeq\mathcal{C}_{\Lambda'}$ where $\mathcal{C}=\mathcal{C}_{(S,M)}$ is the cluster category associated to the surface $(S,M)$. So we have the following commutative diagram:
$$
\begin{tikzpicture}[bij/.style={above,sloped,inner sep=0.5pt}]

\node (1) at (0,0) {$\mathcal{C} $};
\node (2) at (2,0) {$\mathcal{C}$};
\node (3) at (0,1.5) {$ \mathcal{D}^b(\Lambda)$};
\node (4) at (2,1.5) {$ \mathcal{D}^b(\Lambda')$};
\draw [->] (1) edge node[bij] {$\sim$} node[below]{$\scriptstyle f$} (2);
[bij/.style={below,sloped,inner sep=0.5pt}]
\draw [->] (3) edge node[bij] {$\sim$} node[below]{$\scriptstyle F$} (4);
\draw [->] (3)--(1)node[mylabel2]{$\pi$};
\draw [->] (4)--(2)node[mylabel2]{$\pi'$};
\end{tikzpicture}
$$
Now $\pi'\Lambda'=\Delta'$ is the canonical cluster-tilting object in $\mathcal C_{\Lambda'}$. By \cite[Cor. 1.7]{BZ} the cluster-tilting objects of $\mathcal{C}$ are in bijection with the ideal triangulations of $(S,M)$, so one can pass from any cluster-tilting object to another by a sequence of mutations. Let $s$ be a sequence of mutations such that $\mu_s(\pi\Lambda')\simeq f\pi\Lambda=f(\Delta)$. Denote by $T:=\mu_{s}^{L}\Lambda'\in\mathcal{D}^b(\Lambda')$, then 
$$ \pi'(T)= \pi'(\mu_s^L\Lambda')\simeq\mu_s\pi'\Lambda'\simeq f\pi\Lambda\simeq\pi' F\Lambda.$$
So we have the following isomorphism of algebras $\End_{\mathcal{C}}(\pi'T)=\End_{\mathcal{C}}(f\pi\Lambda)\simeq \End_{\mathcal{C}}(\pi \Lambda)$ since $f$ is an equivalence. Furthermore, by Theorem \ref{cor5.4} there is an $\mathbb S_2$-equivalence between the subcategories $\pi'^{-1}(f\pi\Lambda)\subseteq \mathcal{D}^b(\Lambda')$ and $\pi^{-1}(\pi\Lambda)\subseteq\mathcal{D}^b(\Lambda)$, which exactly means that the following algebra isomorphism
$$ \End_\mathcal{C}(\pi' T)=\oplus_{p\in\mathbb Z} \Hom_{\mathcal{D'}}(T, \mathbb{S}_2^{-p}T) \simeq \oplus_{p\in\mathbb Z} \Hom_{\mathcal{D}}(\Lambda, \mathbb{S}_2^{-p}\Lambda)=\End_\mathcal{C}(\pi\Lambda),$$
is a graded equivalence.

Furthermore, from Theorem \ref{thm6.12} we have the following isomorphisms of $\mathbb Z$-graded algebras 
$$ \oplus_{p\in\mathbb Z}\Hom_{\mathcal D'}(T,\mathbb{S}_2^{-p}T)\simeq \jac(\mu_s^L(Q',W',d')), $$ and  $$ \oplus_{p\in\mathbb Z}\Hom_{\mathcal D}(\Lambda,\mathbb{S}_2^{-p}\Lambda)\simeq \jac(Q,W,d) .$$
Thus there is a graded equivalence
$$\jac(\mu_s^L(Q',W',d'))\sim\jac(Q,W,d).$$
Denote by $(\Delta'',d''):=\mu_s^L(\Delta,d)$ and by $(Q'',W'',d'')$ the corresponding graded QP. Then by Proposition \ref{revisited2} there is a homeomorphism of the surface $\Phi:S\to S$ such that $\Phi(\Delta)=\Delta''$ and $w^\epsilon(\Delta,d)=w^\epsilon(\Delta,d''\circ\Phi)$. Since $\mu_s^L(Q',W',d')=(Q'',W'',d'')$ we have that $w^\epsilon(\Delta'',d'')=w^\epsilon(\Delta',d')$ by Lemma~\ref{lemmamutation}. 

\medskip

$(2)\Rightarrow (1)$ 
Denote by $(\Delta'',d''):=(\Phi^{-1}(\Delta'),d'\circ \Phi)$, and by $(Q'',W'',d'')$ the corresponding graded QP.  
There is an isomorphism of graded algebras 
$$ \jac(Q'',W'',d'')\simeq\jac(Q',W',d'),$$
so we have an algebra isomorphism between their degree zero parts $\Lambda''\simeq\Lambda'$. 
Let $s$ be a sequence of mutations such that $\mu_s(\Delta'')=\Delta$. Then  $\mu_s^L(\Delta'',d'')=(\Delta,\delta)$ for some degree map $\delta$, so we get
\begin{align*} w^\epsilon(\Delta,\delta) &= w^\epsilon( \mu_s^L(\Delta'',d''))  \\ & =w^\epsilon(\Delta'',d'') \quad \textrm{ by Lemma~\ref{lemmamutation}}\\ & =w^\epsilon(\Phi^{-1}(\Delta'),d'\circ\Phi) \\ & =w^\epsilon(\Delta,d) \quad \textrm{ by hypothesis (1).} \end{align*} 

So $d$ and $\delta$ are equivalent gradings by Lemma \ref{DRlemma}. By Theorem \ref{AO-derequ}  it follows that $\Lambda$ and $\Lambda''\simeq \Lambda'$ are derived equivalent.
\end{proof}

\begin{remark}
We always assume that $S$ is not a disc. In case of a disc, then the fundamental group is trivial, and the curve $c_1$ is contractible. In that case, any surface algebra is derived equivalent to the path algebra of an acyclic quiver of type $A_p$ where $p$ is the number of marked points on the boundary of the disc \cite[Cor. 3.16]{AO-tilde}.\end{remark}
%
%
%
%

\section{Genus zero case}\label{genuszerosection}

In this section we focus on the case where the surface $S$ has genus zero. In this case it is quite straightforward to apply Theorem~\ref{maintheorem} as the action of a homeomorphism on the fundamental group is easy to understand . More precisely we have the following result.
\begin{lemma}
Let $(S,M)$ be an oriented surface of genus $0$ with $b$ boundary components $B_1,\ldots, B_b$. For each $i=1,\ldots, b$ we denote by $p_i$ the cardinal $\sharp M\cap B_i$ and by $c_i$ a simple closed curve around $B_i$ and following the orientation of $S$. Let $\Phi:S\to S$ be an orientation preserving homeomorphism of $S$ with $\Phi(M)=M$. Then there exists a permutation $\sigma\in \mathfrak{S}_b$ with $p_{\sigma i}=p_i$ for any $i=1,\ldots, b$ and such that $\Phi( c_i)$ is isotopic to $c_{\sigma i}$.
\end{lemma}

\begin{proof}
The curve $c_i$ cuts the surface $S$ into two connected components, one being an annulus with one boundary component with $p_i$ marked points. The simple closed curve $\Phi(c_i)$ should satisfy the same properties, so it is isotopic to some $c_j$ with $p_i=p_j$.
\end{proof}

Combining this lemma together with Theorem \ref{maintheorem} we get the following.
\begin{corollary}\label{genuszerocorollary}
Let  $\Lambda=(S,M,\Delta, d)$  and $\Lambda'=(S,M,\Delta',d')$ be two surface algebras with $S$ of genus zero. Then $\Lambda$ is derived equivalent to $\Lambda'$ if and only if  there exists a permutation $\sigma\in \mathfrak{S}_b$ with $p_{\sigma i}=p_i$ for any $i=1,\ldots, b$ with $d'(\bar{c_i})=d(\bar{c}_{\sigma i})$.
\end{corollary}

\begin{remark}
In the case where $g=0$ and $b=2$, that is when $(S,M)$ is an annulus, this result has already been proved in \cite{AO-tilde}. Indeed one easily checks that the weight $w(d)$ defined in \cite[Def. 6.16]{AO-tilde} coincide with $d(\overline{c_1})$. Moreover by Lemma~\ref{propcb} we have $d(\overline{c}_1)+d(\overline{c}_2)=0$. By Corollary \ref{genuszerocorollary} the derived equivalence class of $\Lambda$ is determined by $d(\overline{c}_1)$ except if $p_1=p_2$ in which case it is determined by $|d(\overline{c}_1)|$, which is exaclty Theorem 4.5 and 4.7 of \cite{AO-tilde}.
\end{remark}
\begin{example}
We now look at a surface of genus $0$ with three boundary components $B_1,\ B_2$ and $B_3$, where $B_1$ and $B_2$ has one marked point each and $B_3$ has two marked points. Let $\Delta$ and $\Delta'$ be the triangulations shown in Figure \ref{triangulation1} and \ref{triangulation2}, respectively. Denote by $(Q,W)$, and $(Q',W')$ their respective quiver with potential.  
\begin{figure}[h!]
\centering
\scalebox{1.0}{
\begin{tikzpicture}[>=stealth]

\draw [thick,domain=0:360] plot ({0.35*cos(\x)}, {0.35*sin(\x)});
\draw [thick,domain=0:360] plot ({4+0.35*cos(\x)}, {0.35*sin(\x)});
\draw [thick,domain=0:360] plot ({2+0.35*cos(\x)}, {-2.5+0.35*sin(\x)});

\node at (0,0){$B_1$};
\node at(4,0){$B_2$};
\node at(2,-2.5){$B_3$};

\node (1) at(0.35,0){$\bullet$};
\node (2) at(3.65,0){$\bullet$};
\node (3) at(2.3031,-2.325){$\bullet$};
\node (4) at(1.65,-2.5){$\bullet$};

\draw (0.35,0) -- (3.65,0) node[mylabel2] {2};
\draw (3.65,0) -- (2.3031,-2.325) node[mylabel2]{6};
\draw (0.35,0) -- (2.3031,-2.325) node[mylabel2]{3};
\draw (0.35,0) -- (1.65,-2.5) node[mylabel2]{4};
\draw (2.3031,-2.325) ..controls (3.5,-5) and (-0.3,-2.7) ..(0.35,0) node[mylabel2]{5};
\draw (3.65,0) ..controls (4,2) and (7,0) .. (2.3031,-2.325)node[mylabel2]{7};
\draw (0.35,0) ..controls (-0.25,-2.5) and (-3,3)  .. (3.65,0)node[mylabel2]{1};

\node (5) at (5.5,-1) {$5$};
\node (1) at (7,-1) {$1$};
\node (2) at (8.5,-1) {$2$};
\node (3) at (10,-1) {$3$};
\node (4) at (8.5,0) {$4$};
\node (6) at (8.5,-2.5) {$6$};
\node (7) at (7,-2.5) {$7$};

\draw [->] (2) -- (1)node[mylabel2]{$\alpha_{21}$};
\draw [->] (5) -- (1)node[mylabel2]{$\alpha_{51}$};
\draw [->] (1) -- (7)node[mylabel2]{$\alpha_{17}$};
\draw [->] (7) -- (5)node[mylabel2]{$\alpha_{75}$};
\draw [->] (2) -- (3)node[mylabel2]{$\alpha_{23}$};
\draw [->] (3) -- (4)node[mylabel2]{$\alpha_{34}$};
\draw [->] (4) -- (5)node[mylabel2]{$\alpha_{45}$};
\draw [->] (6) -- (2)node[mylabel2]{$\alpha_{62}$};
\draw [->] (6) -- (7)node[mylabel2]{$\alpha_{67}$};
\draw [->] (3) -- (6)node[mylabel2]{$\alpha_{36}$}; 
 
\end{tikzpicture}}
\caption{The triangulation $\Delta$ and the associated quiver $Q$.} \label{triangulation1}
\end{figure}
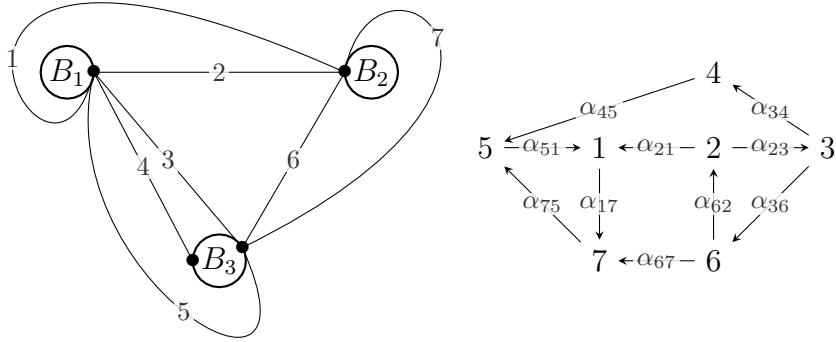

For the quiver $Q$ we have that $W=\alpha_{23}\alpha_{62}\alpha_{36} + \alpha_{51}\alpha_{75}\alpha_{17} $. Consider the following admissible cuts on $\jac(Q,W)$:
\begin{itemize}
\item $d_0(\alpha_{62})=d_0(\alpha_{17})=1$ 
\item $d_1(\alpha_{36})=d_1(\alpha_{75})=1$ 
\item $d_2(\alpha_{23})=d_2(\alpha_{17})=1$ 
\item $d_3(\alpha_{23})=d_3(\alpha_{51})=1$ 
\end{itemize} 
If we consider $\epsilon=(c_1,c_2,c_3)$ where $c_i$ is a simple closed curve surrounding $B_i$ following the orientation of $S$, we get \begin{itemize}\item $\bar{c}_1=\alpha_{23}+\alpha_{34}+\alpha_{45}+\alpha_{51}-\alpha_{21}$, 
\item $\bar{c}_2=\alpha_{62}+\alpha_{21}+\alpha_{17}-\alpha_{67}$
\item $\bar{c}_3=\alpha_{67}+\alpha_{36}-\alpha_{34}-\alpha_{45}+\alpha_{75}$.
\end{itemize}

Hence one immediately checks \[w^\epsilon(d_0)=(0,2,0),\ w^\epsilon(d_1)=(0,0,2),\  w^\epsilon(d_2)=(1,1,0)\ \textrm{and}\ w^\epsilon(d_3)=(2,0,0).\]
In this example we denote the surface algebra associated with $(\Delta_i,d_i)$ (resp. $(\Delta_i',d'_i)$) by $\Lambda_i$ (resp. ($\Lambda_i'$)). 
By Corollary~\ref{genuszerocorollary} $\Lambda_0$ is derived equivalent to $\Lambda_3$ since there is a homeomorphism of $S$ exchanging the roles of $B_1$ and $B_2$. Neither of these algebras are derived equivalent to $\Lambda_1$, since there is no such homeomorphism as $B_3$ has $2$ marked points while $B_1$ and $B_2$ has only one marked point each.
%
%

Furthermore, for $Q'$ we have that $W'=\alpha_{17}\alpha_{51}\alpha_{75} + \alpha_{37}\alpha_{63}\alpha_{76}+\alpha_{43}\alpha_{54}\alpha_{35}$. 
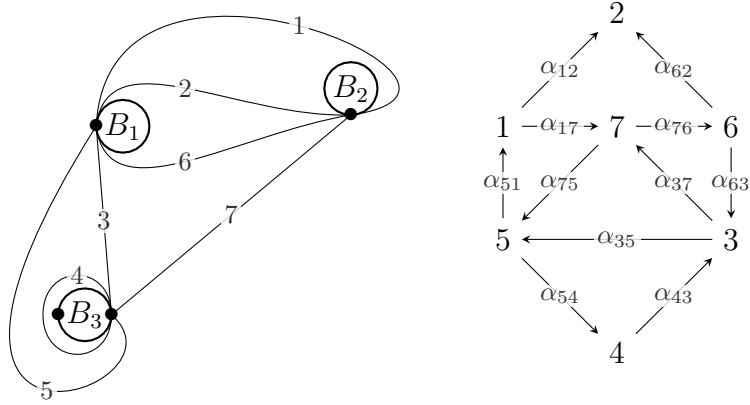
\begin{figure}[h!]
\centering
\scalebox{1.0}{
\begin{tikzpicture}[>=stealth]

\draw [thick,domain=0:360] plot ({0.35*cos(\x)}, {0.35*sin(\x)});
\draw [thick,domain=0:360] plot ({3+0.35*cos(\x)}, {0.5+0.35*sin(\x)});
\draw [thick,domain=0:360] plot ({-0.5+0.35*cos(\x)}, {-2.5+0.35*sin(\x)});

\node at (0,0){$B_1$};
\node at(3,0.5){$B_2$};
\node at(-0.5,-2.5){$B_3$};

\node (1) at(-0.35,0){$\bullet$};
\node (2) at(3,0.15){$\bullet$};
\node (3) at(-0.15,-2.5){$\bullet$};
\node (4) at(-0.85,-2.5){$\bullet$};

\draw (-0.35,0) -- (-0.15,-2.5) node[mylabel2]{3};
\draw (-0.15,-2.5) -- (3,0.15) node[mylabel2]{7};
\draw (-0.35,0) ..controls (-0.2,-1.2) and (1.5,-0.1).. (3,0.15) node[mylabel2]{6};
\draw (-0.35,0) ..controls (-0.2,1.2) and (1.5,0.1).. (3,0.15) node[mylabel2]{2};
\draw (-0.35,0) ..controls (-0.2,3) and (5.5,0.5).. (3,0.15) node[mylabel2]{1};
\draw (-0.35,0) ..controls (-3.5,-5) and (1,-3.5).. (-0.15,-2.5) node[mylabel2]{5};
\draw (-.15,-2.5) ..controls (-0.15,-1.8) and (-1.05,-1.8).. (-1.05,-2.5) node[mylabel2]{4};
\draw (-1.05,-2.5) ..controls (-1.05,-3.2) and (-0.15,-3.2).. (-.15,-2.5);

\node (1) at (5,0) {$1$};
\node (2) at (6.5,1.5) {$2$};
\node (3) at (8,-1.5) {$3$};
\node (4) at (6.5,-3) {$4$};
\node (5) at (5,-1.5) {$5$};
\node (6) at (8,0) {$6$};
\node (7) at (6.5,0) {$7$}; 
 
\draw [->] (1) -- (7)node[mylabel2]{$\alpha_{17}$};
\draw [->] (1) -- (2)node[mylabel2]{$\alpha_{12}$};
\draw [->] (5) -- (1)node[mylabel2]{$\alpha_{51}$};
\draw [->] (5) -- (4)node[mylabel2]{$\alpha_{54}$};
\draw [->] (4) -- (3)node[mylabel2]{$\alpha_{43}$};
\draw [->] (3) -- (5)node[mylabel2]{$\alpha_{35}$};
\draw [->] (3) -- (7)node[mylabel2]{$\alpha_{37}$};
\draw [->] (6) -- (3)node[mylabel2]{$\alpha_{63}$};
\draw [->] (7) -- (6)node[mylabel2]{$\alpha_{76}$};
\draw [->] (7) -- (5)node[mylabel2]{$\alpha_{75}$};
\draw [->] (6) -- (2)node[mylabel2]{$\alpha_{62}$};
 
\end{tikzpicture}}
\caption{The triangulation $\Delta'$ and the associated quiver $Q'$.} \label{triangulation2}
\end{figure}
Consider the following degree-$1$ maps on $\jac(Q',W')$:
\begin{itemize}
\item $d'_0(\alpha_{37})=d'_0(\alpha_{35})=d'_0(\alpha_{17})=1$ and $w^\epsilon(d'_0)=(1,1,0)$.
\item $d'_1(\alpha_{43})=d'_1(\alpha_{37})=d'_1(\alpha_{75})=1$ and $w^\epsilon(d'_1)=(0,0,2)$.
\end{itemize}
Then $\Lambda_0'$ is derived equivalent to $\Lambda_2$, see figure \ref{surfacealgebras}, and $\Lambda'_1$ is derived equivalent to $\Lambda_1$.
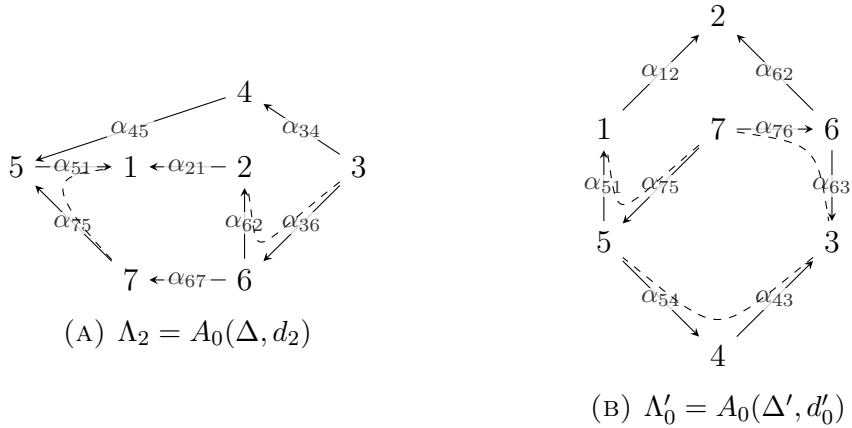
\begin{figure}[h!]
\begin{subfigure}{0.45\textwidth}
\centering
\begin{tikzpicture}[>=stealth]
\node (5) at (5.5,-1) {$5$};
\node (1) at (7,-1) {$1$};
\node (2) at (8.5,-1) {$2$};
\node (3) at (10,-1) {$3$};
\node (4) at (8.5,0) {$4$};
\node (6) at (8.5,-2.5) {$6$};
\node (7) at (7,-2.5) {$7$};
\draw [->] (2) -- (1)node[mylabel2]{$\alpha_{21}$};
\draw [->] (5) -- (1)node[mylabel2]{$\alpha_{51}$};
\draw [->] (7) -- (5)node[mylabel2]{$\alpha_{75}$};
\draw [->] (3) -- (4)node[mylabel2]{$\alpha_{34}$};
\draw [->] (4) -- (5)node[mylabel2]{$\alpha_{45}$};
\draw [->] (6) -- (2)node[mylabel2]{$\alpha_{62}$};
\draw [->] (6) -- (7)node[mylabel2]{$\alpha_{67}$};
\draw [->] (3) -- (6)node[mylabel2]{$\alpha_{36}$}; 
\draw [dashed] (7) ..controls (5.8,-1.2) and (6,-1.1).. (1);
\draw[dashed] (3) ..controls (8.5,-2.2) and (8.7,-2.3) .. (2);
\end{tikzpicture}
\caption{$\Lambda_2=A_0(\Delta,d_2)$}
\label{}
\end{subfigure}
~
\begin{subfigure}{0.45\textwidth}
\centering
\begin{tikzpicture}[>=stealth]
\node (1) at (5,0) {$1$};
\node (2) at (6.5,1.5) {$2$};
\node (3) at (8,-1.5) {$3$};
\node (4) at (6.5,-3) {$4$};
\node (5) at (5,-1.5) {$5$};
\node (6) at (8,0) {$6$};
\node (7) at (6.5,0) {$7$}; 
\draw [->] (1) -- (2)node[mylabel2]{$\alpha_{12}$};
\draw [->] (5) -- (1)node[mylabel2]{$\alpha_{51}$};
\draw [->] (5) -- (4)node[mylabel2]{$\alpha_{54}$};
\draw [->] (4) -- (3)node[mylabel2]{$\alpha_{43}$};
\draw [->] (6) -- (3)node[mylabel2]{$\alpha_{63}$};
\draw [->] (7) -- (6)node[mylabel2]{$\alpha_{76}$};
\draw [->] (7) -- (5)node[mylabel2]{$\alpha_{75}$};
\draw [->] (6) -- (2)node[mylabel2]{$\alpha_{62}$};
\draw [dashed] (3) ..controls (6.4,-2.8) and (6.6,-2.8).. (5); 
\draw [dashed] (7) ..controls (5.1,-1.2) and (5.2,-1.3).. (1);
\draw [dashed] (7) ..controls ( 7.8,-0.2) and (7.8 ,-0.2).. (3);
\end{tikzpicture}
\caption{$\Lambda'_0=A_0(\Delta',d'_0)$}
\label{}
\end{subfigure}

\label{surfacealgebras}
\caption{Two derived equivalent algebras associated with different triangulations.}
\end{figure}

\end{example}

%
%
%
%
\section{The AG-invariant for surface algebras}\label{AGsection}

In \cite{AG} the authors introduced a derived invariant (called AG invariant) for gentle algebras. Moreover, a surface algebra is gentle (see \cite[Proposition 2.8]{DS1}). In this section we relate the AG-invariant of a surface algebra to its weight, and deduce some consequences of this relation.
\subsection{AG invariant and weight}

The AG-invariant, \cite{AG}, is a combinatorial derived invariant for gentle algebras. The invariant is calculated by an algorithm producing ordered pairs of natural numbers. The input of the algorithm is a maximal directed path in $Q$ with no zero relations. The second step is to go from the ending vertex of the input backwards along a path of zero relations as long as possible. This procedure is repeated until the input path appears again after $n$ steps of the algorithm. We then obtain the ordered pair $(n,m)$ where $m$ is the sum of the number of arrows which appeared in zero relations in the $n$ steps of the algorithm. The algorithm is to be repeated until all maximal directed paths without zero relations have appeared in a step of the algorithm. The AG-invariant for an algebra $\Lambda$ is the function $AG(\Lambda): \mathbb{N}^2\rightarrow \mathbb N$ which counts how often each pair $(n,m)$ occurs.

We have the following result from \cite{AG}.
\begin{theorem}
Any two derived equivalent gentle algebras have the same AG-invariant. Furthermore, gentle algebras with at most one cycle in the quiver are derived equivalent if and only if they have the same AG-invariant.
\end{theorem}

If $\Lambda=(S,M,\Delta,d)$ is a surface algebra with $\tau=abc$ an internal triangle such that $d(a)=1$ and $\alpha=(i,j)$ the angle corresponding to $a$ in $\Delta$, then we say that $d$ defines a \textit{local cut between $i$ and $j$} on the marked point incident with both $i$ and $j$. Denote by $n(B_i,\Delta)$ the number of marked points on $B_i$ that are incident to at least one arc of $\Delta$, and let $m(B_i,\Delta)$ be the number of boundary segments on $B_i$ where both marked endpoints are incident with at least one arc of $\Delta$.

\begin{theorem}~\cite{DS1}\label{theoremDS}
Let $\Lambda=(S,M,\Delta,d)$ be a surface algebra. Then the AG-invariant $AG(\Lambda)$ can be computed in the following way;
\begin{itemize}
\item The ordered pairs $(0,3)$ in $AG(\Lambda)$ are in bijection with the internal triangles where all arrows have degree $0$, and there are no ordered pairs $(0,m)$ for $m\neq3$.
\item The ordered pairs $(n,m)$ in $AG(\Lambda)$ with $n\neq 0$ are in bijection with the boundary components of $S$. If $B_i$ is a boundary component, the corresponding ordered pair $(n_i,m_i)$ is given by 
$$ n_i=n(B_i,\Delta)+\ell_i, \ \ m_i=m(B_i,\Delta)+2\ell_i, $$
where $\ell_i$ is the number of local cuts of $(\Delta,d)$ on $B_i$.
\end{itemize}
\end{theorem}
Note that in our setting all internal triangles are cut, and thus there are no ordered pairs $(0,m)$.

The relation between the AG-invariant and the weight defined in Section 2 is given by the following result.

\begin{proposition}\label{propAG}Let $(S,M)$ be an oriented marked surface which is not a disc.
Let $\Lambda=(\Delta,d)$ be a surface algebra and for any $i=1,\ldots,b$ let $c_i$ be a simple closed curve surrounding the boundary component $B_i$ and following the orientation of $S$. Denote by $(n_i,m_i)$ the ordered pair of the AG-invariant of $\Lambda$ corresponding to $B_i$. Then we have
\[ n_i= p_i+d(\bar{c_i}) \quad \textrm{and} \quad m_i=p_i+2d(\bar{c_i}),\]
where $p_i$ is the number of marked points on $B_i$.
\end{proposition}

\begin{proof} The proof is done in several steps. We start with some notation. 
Let $q_i$ be the number of isolated marked points on $B_i$, i.e. the marked points on $B_i$ not incident with any arc of $\Delta$. We denote by $\chi_i$ the number of internal triangles with at least one side homotopic to a segment of $B_i$. 

\medskip

\textit{Step 1: We prove that $n_i=p_i-q_i+\ell_i$ and $m_i=p_i-2 q_i+2\ell_i$.}

\noindent
By definition, we have $n(\Delta,B_i)=p_i-q_i$, so we get the first equality by Theorem \ref{theoremDS}.
Two isolated points cannot be neighbours, and thus each isolated point is incident with exactly two boundary segments.  
Since there are $p_i$ segments on $B_i$ we get $m(\Delta,B_i)=p_i-2q_i$ and the second equality by Theorem \ref{theoremDS}.

\medskip

\textit{Step 2: We prove that $d(\bar{c_i})=\ell_i-\chi_i$.}

\noindent
The degree map $d$ cuts each internal triangle exactly once. We want to compare the cuts counted in $\ell_i$ with the cuts counted in $d(\bar{c_i})$ (with $+$ or $-$ signs). Note that the curve $c_i$ intersects only arcs of $\Delta$ which are not homotopic to the boundary. Recall from Section 2 that there are three kinds of  triangles. Let $\tau$ be an internal triangle. If $\tau$ is uncontractible, then the cut of $\tau$ is counted in $\ell_i$ if and only if it is counted with a $+$ in $d(\bar{c_i})$.  If $\tau$ is a triangle homotopic to the boundary, then its cut is not involved in $d(\bar{c_i})$, and it is in $\ell_i$ if and only if $\tau$ is homotopic to a segment of $B_i$, that is, if $\tau$ has at least one side homotopic to a segment of $B_i$. Finally, assume that $\tau$ is a triangle based on the boundary. If the cut is opposite to the side homotopic to the boundary, then it counts in $\ell_i$ if and only if it counts with a $+$ in $d(\bar{c_i})$. And moreover, it counts with a $-$ in $d(\bar{c_i})$ if and only if $\tau$ is based on $B_i$, that is when $\tau$ has at least one side homotopic to a segment of $B_i$. If the cut is in some other angle, then it does not count in $d(\bar{c_i})$ and it counts in $\ell_i$ if and only if $\tau$ is based on $B_i$. Thus the cuts that  appear  in the difference $\ell_i-d(\bar{c_i})$ are exactly the cuts in triangles with one side homotopic to a segment of $B_i$. 

\medskip

\textit{Step 3 is a technical lemma. }
\begin{lemma}\label{lemmadisc}
Let $\Delta'$ be an ideal triangulation of a disc with $p\geq 4$ marked points. Then we have $$I(\Delta')+2=q(\Delta')$$ where $I(\Delta')$ is the number of internal triangles of $\Delta'$ and $q(\Delta')$ is the number of isolated points in $\Delta'$.
\end{lemma}

\begin{proof}
The proof is done by induction on $p$. For $p=4$ it is clear. Assume the result for some $p\geq 4$ and let $\Delta'$ be an ideal triangulation of the disc with $p+1$ marked points. There is at least one marked point $A$ which is isolated. Then the two neighbours of $A$ are linked by an arc $\gamma$ of $\Delta'$. If we remove $A$ and $\gamma$, we obtain an ideal triangulation $\Delta''$ of the disc with $p$ marked points. If $\gamma$ was not a side of an internal triangle, then exactly one neighbour of $A$ is isolated in $\Delta''$, so we get $I(\Delta')=I(\Delta'')$ and $q(\Delta')=q(\Delta'')$. If $\gamma$ is the side of an internal triangle in $\Delta$ then we have $I(\Delta')=I(\Delta'')+1$ and $q(\Delta')=q(\Delta'')+1$. We conclude by the induction hypothesis.
\end{proof}

\textit{Step 4: We prove $q_i=\chi_i$.}

Let $\tau$ be a triangle based on $B_i$. Denote by $q(\tau)$ the number of isolated points which are in the segment homotopic to the base of $\tau$. This segment is unique since $S$ is not a disc. The base of $\tau$ cuts out a disc of the surface $S$. This disc has at least three marked points: the endpoints of the base of $\tau$ which has to be distinct, and at least one point in the segment of $B_i$. This disc together with $\tau$ is again a disc with at least $4$ marked points. Denote by $\Delta_\tau$ the triangulation $\Delta$ restricted to that disc.  Since the vertex of $\tau$ opposite to the base is isolated in $\Delta_\tau$, the number of isolated points of $\Delta_\tau$ is $q(\tau)+1$. Therefore by Lemma~\ref{lemmadisc} we have \[q(\tau)=I(\Delta_\tau)+1.\] Taking the sum over the set of the internal triangles based on $B_i$ we obtain: 
\begin{equation}\label{eqqchi}\sum_{\tau }q(\tau)=\sum_{\tau }I(\Delta_\tau)+\sharp\{\textrm{internal triangles based on }B_i\}.\end{equation}
We then have to check that each isolated point on $B_i$ appears exactly once on the left hand side of this equality, and that each internal triangle homotopic to $B_i$ appears exactly once on the right hand side of the equality. 

Let $\tau_0$ be a triangle which is homotopic to $B_i$. We construct a unique triangle $\tau$ based on $B_i$ such that $\tau_0$ is a triangle in $\Delta_{\tau}$. Since $S$ is not a disc, each side of $\tau_0$ cuts the surface $S$ into a disc and a surface which is not a disc. Let $\gamma_0$ be the unique side of $\tau_0$ such that $\tau_0$ is in the disc cut out by $\gamma_0$. Let $\tau_1$ be the other triangle containing $\gamma_0$. If $\tau_1$ is based on $B_i$ then $\tau=\tau_1$ satisfies $\tau_0\in \Delta_{\tau}$. If not, $\tau_1$ is a triangle homotopic to $B_i$. Let $\gamma_1$ be the side of $\tau_1$ such that $\tau_1$ is in the disc cut out by $\gamma_1$. The arc $\gamma_1$ is different from $\gamma_0$, since $\tau_1$ is not in the disc cut out by $\gamma_0$. So by induction on the number of triangles, there exists $n\geq 1$ such that $\tau=\tau_n$ is based on $B_i$, and we have $\tau_0\in \Delta_{\tau}$.  Note that by this construction the triangle $\tau$ is uniquely determined.  So we have
$$\sum_\tau I(\Delta_\tau)=\sharp \{\textrm{internal triangles homotopic to }B_i\}$$ and the right hand side of \ref{eqqchi} equals $\chi_i$.

Now let $A$ be an isolated point. Then $A$ is the vertex of one triangle $\tau_0$ homotopic to $B_i$. By the previous construction, there exists a unique $\tau$ based on $B_i$ such that $\tau_0\in \Delta_\tau$. So $A$ is an isolated point of $\Delta_\tau$.  Hence the right hand side of \ref{eqqchi} is the number of isolated points on $B_i$, that is $q_i$.   

\medskip

Combining the steps together we get the proposition.

\end{proof}

\begin{remark}
If $S$ is a disc, then for any $(\Delta,d)$ the quiver of the surface algebra is a tree. In that case, by \cite[Section 7 (1)]{AG} the surface algebra is derived equivalent to $A_n$ and the AG-invariant is $(n-1,n+1)$. There is no weight since the fundamental group of $S$ is trivial so the above property does not hold. 
\end{remark}

\begin{corollary}
Let $(S,M)$ be a surface of genus $0$ and let $\Lambda=(S,M,\Delta,d)$ and $\Lambda'=(S,M,\Delta',d')$ be surface algebras. Then 
$$ \mathcal{D}^b(\Lambda) \simeq \mathcal{D}^b(\Lambda')  \Leftrightarrow AG(\Lambda)=AG(\Lambda').$$
\end{corollary}

\begin{proof}
Denote by $B_1,\ldots, B_b$ the boundary components of $S$ and by $p_i$ the number of marked points on $B_i$.
For any $i=1,\ldots,b$ let $c_i$ be a simple closed curve surrounding $B_i$ and following the orientation of $S$. Then by Proposition \ref{propAG} $AG(\Lambda)=AG(\Lambda')$ if and only if there exists a permutation  $\sigma\in\mathfrak{S}_b$ such that  $p_{\sigma(i)}=p_i$ and  $d(\bar{c}_{\sigma(i)})=d'(\bar{c_i})$ for any $i$. We then conclude by Corollary \ref{genuszerocorollary}.
\end{proof}

\begin{remark}
When $\Delta=\Delta'$ in the genus zero case, we obtain that $\Lambda$ and $\Lambda'$ are derived equivalent if and only if there exists a permutation $\sigma$ of the boundary components such that $p_{\sigma i}=p_i$ and with $\ell_{\sigma i}=\ell'_i$ where $\ell_i$ is the number of cuts on $B_i$ corresponding to the grading $d$ and $\ell_i'$ corresponding to the grading $d'$. One direction of this result has been already proved by~\cite{DR2}. 
\end{remark}

\begin{remark} When $g\geq 1$ there are surface algebras with the same AG-invariant which are not derived equivalent (see \cite{genus1}).
\end{remark}

\subsection{Auslander-Reiten quiver of the derived category of a surface algebra}

In \cite{AG} the authors give an interpretation of the AG-invariant of a gentle algebra $\Lambda$ in terms of the Auslander-Reiten quiver of the category $\smod$ where $\hat{\Lambda}$ is the repetitive algebra of $\Lambda$ (see \cite{happel} for definition). Roughly speaking $AG(\Lambda)$ describe the characteristic components of the AR-quiver, that are the connected components which are tubes (coming from string modules) and of type $\mathbb{Z}A_\infty$.  More precisely they prove the following result.

\begin{theorem}\cite{AG}
Let $\Lambda$ be a gentle algebra. Then $AG(\Lambda)(n,n)=q\geq 1$ if and only if there exist $q$  tubes (coming from strings) of rank $n$ in the AR-quiver of $\smod$ and these tubes are stable under the action of the syzygy functor $\Omega$. 

Moreover for $m\neq n$ we have $AG(n,m)=q\geq 1$ if and only if there exist $q$ families consisting of $|m-n|$ components of type $\mathbb{Z}A_\infty$ in the AR-quiver of $\smod$. Each family is an orbit under the action of the syzygy functor $\Omega$, and for each indecomposable object in these components we have $\Omega^{n-m} M=\tau^n M$ where $\tau$ is the AR-translation of $\smod$.
\end{theorem}   

Applying this result to our setup, we obtain some information on the characteristic components of the category $\mathcal{D}^b( \Lambda)$ for $\Lambda$ a surface algebra. Indeed since $\Lambda$ has finite global dimension there is a triangle equivalence between the categories $\smod$ and $\mathcal{D}^b(\Lambda)$ \cite[Chapter 2.4]{happel}. 

\begin{corollary}\label{corAR}
Let $\Lambda=(S,M,\Delta,d)$ be a surface algebra where $S$ is not a disc. For any $i=1,\ldots, b$, let $c_i$ be a direct simple closed curve surrounding the boundary component $B_i$, and denote by $w_i=d(\bar{c_i})$. Then the characteristic components of the AR-quiver of $\mathcal{D}^b(\Lambda)$ are described as follows:
\begin{itemize}
\item for each $i$ such that $w_i=0$ there exists a tube of rank $p_i$, which is stable under the action of the shift functor $[1]$;
\item for each $i$ with $w_i\neq 0$ there exists a family of $|w_i|$ components of type $\mathbb{Z}A_\infty$ which is an orbit for the action of the shift functor $[1]$. For each indecomposable object in these components we have $M[w_i]=\tau^{p_i+w_i} M$. 
\end{itemize}
\end{corollary}

\begin{remark}
When $g=1$ and $b=2$, that is when $(S,M)$ is an annulus, then Corollary \ref{corAR} has already been proved in \cite[Cor. 5.5]{AO-tilde}. In that case we have $w_1=d(\bar{c}_1)=-d(\bar{c_2})=-w_2$. If $w_1=w_2=0$ then $\Lambda$ is derived equivalent to the path algebra of a quiver of type $\widetilde{A}_{p_1,p_2}$ so has 2 tubes of rank $p_1$ and $p_2$. And if $w_1\neq 0$ then there are $|w_1|+|w_2|=2|w_1|$ components of type $\mathbb{Z}A_\infty$. The proof in \cite{AO-tilde} does not use the AG-invariant but the fact that the cluster category of $(S,M)$ is equivalent to the acyclic cluster category of type $\widetilde{A}$. 
\end{remark}

From Corollary \ref{corAR} we deduce some information on the characteristic components of the AR-quiver of the cluster category $\mathcal{C}_{(S,M)}$.  In particular we obtain another proof of a Theorem due to Br\"ustle and Zhang in \cite{BZ}.

\begin{theorem}
Let $(S,M)$ be a surface with marked points which is not a disc. Then for each boundary component $B_i$ (with $p_i$ marked points) there exists a tube of rank $p_i$ in the AR-quiver of $\mathcal{C}_{(S,M)}$. 
\end{theorem}

\begin{proof} As in section 2, we denote by $\mathbb{S}_2:=\mathbb{S}[-2]$ the composition of the Serre functor $\mathbb{S}$ of $\mathcal{D}^b(\Lambda)$ with the shift functor $[-2]$.
Let $i=1,\ldots, b$ with $w_i\neq 0$. Let $C, C[1] \ldots, C[w_i-1]$ be the corresponding family of components of the AR-quiver of $\mathcal{D}^b(\Lambda)$ given by Corollary~\ref{corAR}. Then we have $$\mathbb{S}_2^{-1}(C)=\mathbb{S}^{-1}C[2]=\tau^{-1}C[1]=C[1],$$ since $C$ is stable under the action of $\tau$. Thus all these components are in the same $\mathbb{S}_2$-orbit and we get \[\pi (C)=\pi(C[1])=\ldots=\pi(C[w_i])\] where $\pi$ is the functor $\pi:\mathcal{D}^b(\Lambda)\to\mathcal{D}^b(\Lambda)/\mathbb{S}_2$. Moreover for any $M\in C$ we have  
\begin{align*} 
\tau^{p_i} M &= \tau^{-w_i} (\tau^{p_i+w_i}M)= \tau^{-w_i}(M[w_i])\\ & =\mathbb{S}^{-w_i} (M[w_i]) [w_i]=\mathbb{S}^{-w_i} M[-2w_i]=\mathbb{S}_2^{-w_i} M.
\end{align*}
Hence $\tau^{p_i}M$ and $M$ have the same image in the orbit category $\mathcal{D}^b(\Lambda)/\mathbb{S}_2$. So $\pi(C)$ is a tube of rank $p_i$. We conclude by using the fact that the orbit category $\mathcal{D}^b(\Lambda)/\mathbb{S}_2$ is a full subcategory of $\mathcal{C}_{(S,M)}$.

If $w_i=0$ the proof is similar.

\end{proof}


\end{document}